\documentclass[10pt]{article}

\def\ds{\displaystyle}

\usepackage[colorlinks,allcolors=red]{hyperref}

\usepackage{amsfonts}
\usepackage{graphicx}
\usepackage{psfrag}
\usepackage{color}
\usepackage{amsmath,amsfonts,latexsym,amsbsy,amssymb,amsthm,color,enumitem}

\def\ds{\displaystyle}

\def\R{\mathbb{R}}

\def\bega{\begin{array}}
\def\enda{\end{array}}
\def\begi{\begin{itemize}}
\def\endi{\end{itemize}}

\def\bel{\begin{equation}\label}
\def\eeq{\end{equation}}
\def\sqr#1#2{\vbox{\hrule height .#2pt
\hbox{\vrule width .#2pt height #1pt \kern #1pt
\vrule width .#2pt}\hrule height .#2pt }}

\newtheorem{theorem}{Theorem}[section]

\newtheorem{definition}[theorem]{Definition}
\newtheorem{remark}[theorem]{Remark}
\newtheorem{lemma}[theorem]{Lemma}
\newtheorem{proposition}{Proposition}[section]

\definecolor{forestgreen}{rgb}{0.13, 0.55, 0.13}

\begin{document}

\title{\bf  On the global controllability of scalar conservation laws with  boundary and source controls}

\author{Fabio Ancona$^{(1)}$ and Khai T. Nguyen$^{(2)}$\\ 
\\
 {\small $^{(1)}$ Dipartimento di Matematica ``Tullio Levi-Civita'', Universit\`a di Padova, }\\  
  {\small $^{(2)}$ Department of Mathematics, North Carolina State University }\\ 
 \\  {\small E-mails: ~ancona@math.unipd.it,  ~ khai@math.ncsu.edu}
 }

\date{\today}

\maketitle

\begin{abstract}
We provide global and semi-global controllability results for hyperbolic conservation laws on a bounded domain,
with a general (not necessarily convex) 
flux and a time-dependent source term acting as a control.
The results are achieved for, possibly critical, both continuously differentiable states and BV states.
The proofs are based on a combination of the return method and on the analysis of the Riccati equation
for the space derivative of the solution.
\vspace{0.2in}

\noindent
{\bf Key  words.} conservation laws, source control, global exact controllability, return method
\vspace{7pt}

\noindent
{\bf AMS Mathematics Subject Classification.} 
35L65, 
35Q93,
93B05,
93C20.
\end{abstract}

\section{Introduction and main results} \label{sec:int}
We are concerned with the problem of controllability of a one space-dimensional scalar conservation law
on a bounded domain
\begin{equation}\label{eq:claw-homog}
\partial_t u+\partial_x f(u)=0\,,~~~~t>0,\ x\in [a,b]\,,
\end{equation}
where $u=u(t,x)$ is the state variable 
 and the flux function
$f:I \mapsto\R$ is a smooth map
defined on some open interval $I \subseteq \R$.
Most of the literature concerning the controllability of hyperbolic partial differential equations
analyzes the states $\psi\doteq u(T,\cdot)$ that can be reached at a fixed time $T>0$,
through the influence of boundary controls
 acting at the end points $\{a,b\}$,
when an initial condition is given
\begin{equation}\label{eq:datum}
u(0,x)~=~\overline{u}(x)\,,~~~~x\in [a,b]\,.
\end{equation}
In the case of  conservation laws~\eqref{eq:claw-homog} with a strictly convex flux $f$,
Ancona and Marson~\cite{ancmar1,ancmar2} 
and Adimurthi et al.~\cite{adghogo} 
established a characterization of the rechable states with boundary controls.
A similar characterization of approximately rechable states for the Burgers equation
was provided by Horsin~\cite{hor}.
From these results
it follows that,
if we start with a general initial data $\overline u\in {\bf L^\infty}([a,b])$,
 the profiles $\psi$ that are attainable at a time $T>0$
with boundary controls
at $x=a$ and $x=b$,
 are only those which satisfy suitable Olenik-type inequalities,
provided that 
\begin{equation}
\label{bdry-control-time}
T\geq \overline T\doteq \max\left\{\sup_{{x\in (a,b)}}\frac{x-a}{\lfloor f'\circ \psi(x)\rfloor_+}, \ \sup_{{x\in (a,b)}}\frac{b-x}{\lfloor f'\circ \psi(x)\rfloor_-}\right\},
\end{equation}
where $\lfloor a \rfloor_-\doteq \max\{-a,0\}$, $\lfloor a \rfloor_+\doteq\max\{a,0\}$
denote the negative and positive part, respectively, of $a\in\R$.
\vspace{-5pt}

For conservation laws with general nonconvex flux, Leautaud~\cite{Leau}
proved the attainability in finite time of  constant states, employing boundary controls, while 
Andreianov, Donadello and Marson~\cite{andoma} derived sufficient conditions for the 
reachability  of (non constant) states in the case of a nonconvex flux
with a single inflection point, where one regards as controls the initial data.
All these results show, in particular,  that conservation laws are not 
exactly controllable in finite time to  {\it critical states} (with vanishing characteristic speed).

\vspace{-5pt}

Here, in the same spirit of Chapouly~\cite{Chap} and Perrollaz~\cite{Perrollaz}, we wish to investigate how the effect of a control acting through 
a time dependent source term on the right-hand side
of~\eqref{eq:claw-homog},
in combination with the boundary controls,  allows to: establish  global controllability results;
  achieve the reachability of a broader class of states (including critical states); realize the exact controllability to
 such states in a   shorter time than the one required when employing only boundary controls.

Namely, we shall investigate the exact controllability problem
for a balance law 
\begin{equation}\label{eq:blaw}
\partial_t u+\partial_x f(u)=h(t)\,,~~~~t>0,\ x\in [a,b]\,,
\end{equation}
where we regard as controls both the boundary data acting at the end points $\{a,b\}$
of the domain, and the source term $h$ depending only on time.
We recall that there are two possible settings within which to study this problem.
The first possibility is to consider classical solutions (i.e. Lipschitz continuous
functions that satisfy the equation
almost everywhere), 
assuming that the source and the boundary controls are  regular functions as well. 
The other {possibility} is to consider weak (distributional) solutions which satisfy an entropy
admissibility criterium,
which are natural in this framework since in general classical  solutions of~\eqref{eq:claw-homog} 
develop discontinuities in finite time because of the nonlinearity of the equation. 

In the first setting Chapouly~\cite{Chap} showed that, when $f(u)=u^2/2$,  
for every $T>0$ 
one can drive in time $T$ any preassigned continuously differentiable initial data $\overline u$ to any continuously differentiable target state  
$\psi$
with a classical 
solution of~\eqref{eq:blaw}, using suitable source $h(t)$
and boundary controls at $x=a$, $x=b$.
In the same setting, for quasilinear hyperbolic systems,  
 local~\cite{LiRaoChar,LiYu,zhuang,zhulirao} and global (in the linearly degenerate case)~\cite{wang} controllability results for $C^1$ states
were established employing boundary and distributed controls on the source that depend on both $(t,x)$ variables.
In the second setting and for general strictly convex flux $f$, Perrollaz~\cite{Perrollaz} provided sufficient conditions 
for the reachability (in arbitrarly small time) of a state $\psi\in BV([a,b])$  
with boundary and source controls, through  entropy weak 
solutions of~\eqref{eq:blaw}.
 In a related result Corghi and Marson~\cite{xor-mar} established a
 characterization of the attainable set for scalar strictly convex balance laws 
 evolving on the 
 whole real line,  with  the source term (depending on both space and time) regarded as a control.

In this paper we will first establish the global and semi-global controllability of continuously differentiable states for a conservation law with a general smooth flux,
when time dependent source and boundary controls are acting on the equation.
Next, in the case of convex (non necessarily strictlly convex) 
conservation laws,
we will provide BV bounds on the smooth source control and on the 
$C^1$ solution connecting
an initial datum $\overline u$ to a terminal state $\psi$,
in terms of the positive variation of $\psi$ and of the negative variation of $\overline u$.
Finally, relying on such BV bounds, we will  show the reachability in finite time of  states
$\psi\in BV([a,b])$  that satisfies one-sided Lipschitz estimates similar to those stated in~\cite{Perrollaz}.
The advantage of this construction is that we obtain the source control
and the corresponding solution as limit of regular solutions which
are easier to handle than the piecewise constant
 front tracking solutions employed in~\cite{Perrollaz}.
In fact, we rely on the approach developed in the present paper to address similar problems of global controllability 
for diagonal systems of conservation laws in the forthcoming paper~\cite{anc-khai}.

Control problems for conservation laws arise in many different applications including: vehicular  traffic models~\cite{ACCG1,
ACCG2,CGR,DMPR},
 oil reservoir simulation and sedimentation models~\cite{andoma}, supply chain~\cite{GHK,LAHR}, gas dynamics~\cite{GL-survey}.
In practice a time dependent source control can be  viewed as a control parameter acting
on the flux function of the conservation law letting vary its flux capacity. 
We refer to~\cite{Perrollaz} for a discussion  of various models where source controls naturally appear to
govern the dynamics of the corresponding balance law.

Since we are assuming to have full control on both endpoints $\{a, b\}$ of the domain, 
and because boundary conditions for nonlinear hyperbolic equations are quite involved 
(e.g. see~\cite{BiSpi,Serre}), it will be simpler 
to reformulate the controllability problem in an undetermined form
(where the boundary data are not explicitely prescribed). 
Therefore, given an initial state $\overline u$ and a terminal state $\psi$,
we will rephrase the problem 
of steering~\eqref{eq:claw-homog} from $\overline u$ to $\psi$
via boundary and source controls, into the equivalent one of determining 
a time dependent source  $h=h(t)$ and a solution of~\eqref{eq:blaw}
that satisfies~\eqref{eq:datum} together with the terminal condition 
\begin{equation}\label{terminal-datum}
u(T,x)~=~\psi(x)\,,~~~~x\in [a,b]\,.
\end{equation}
The corresponding boundary controls  can  be recovered afterwards by taking the traces of $u$ at $x=a$ and $x=b$.

Before stating the main results, we recall the definition of entropy admissible weak solutions.
An entropy/entropy flux pair for the equation~\eqref{eq:blaw} is a couple of continuously differentiable maps
$(\eta,q):I \to\R$, that satisfy $D\eta(u)\cdot Df(u)=Dq(u)$ for all $u\in\ I$.
Observe that, in particular, $(\eta,q)=(\pm Id, \pm f(u))$ provide two entropy/entropy flux pairs.
Then we shall adopt the following definition.
\begin{definition}A function $u:[0,T]\times[a,b]\to I$ is called an  \emph{entropic weak solution} of~\eqref{eq:blaw}, \eqref{eq:datum} on $[0,T]\times[a,b]$, if it is a
continuous function from $[0,T]$ into ${\bf L}^1([a,b]; I)$, which assumes almost everywhere the initial datum\eqref{eq:datum}, 
%
%
and that is an entropy admissible distributional solution of~\eqref{eq:blaw} on $(0,T)\times (a,b)$, i.e. such that
for any entropy / entropy flux pair $(\eta,q)$, with $\eta$ convex, there holds
$$
\int_0^T\int_a^b \Big\{\eta(u(t,x))\partial_t\varphi(t,x)+q(u(t,x))\partial_x\varphi(t,x)+\eta'(u(t,x)) h(t)\cdot \varphi(t,x)\Big\}\,dx\,dt\geq 0\,,
$$
for all test functions $\varphi\in {C}_c^1$, 
$\varphi\geq 0$,  with compact support in $]0,T[\,\times\,]a,b[$.
\end{definition}


Our first results concern the global controllability of continuously differentiable states. 
Throughout the paper, for any continuously differentiable map $\varphi: J\to \R$, defined on some interval $J\subset \R$,
we shall adopt the notation 
\begin{equation}
\|\varphi\|_{C(J)}\doteq \sup\{|\varphi(x)|:\, x\in J\}.
\end{equation}
Moreover,  to estimate the maximal speed of the characteristics with which 
can travel an initial data taking values in a given set $J'\subseteq J$,
we introduce
the quantities
\begin{equation}
\label{norm1-[Delta-def}
{[| \varphi|]_{J'}}\doteq  \sup_{\{k\,\,|\, J'+k\,\subseteq J\}}\, \inf_{u\in J'}
\big|\Delta \varphi(u;k)\big|\,,
\qquad
{\Delta \varphi(u;k)\doteq \dfrac{\varphi(u+k)-\varphi (u)}{k},}
=\dfrac{\ds\int_0^k\!\!\varphi'(u+v)~dv }{k}\,,
\end{equation}
and, for every $\varepsilon>0$,
\begin{equation}
\label{argsup}
\arg\text{sup}{[| \varphi|]_{J',\varepsilon}}\doteq 
\begin{cases}
\inf\Big\{k\geq 0\,\,|\, J'+k\,\subseteq J,&   \big|\Delta \varphi(u;k)\big|>{[| \varphi|]_{J'}}-\varepsilon
\ \ \forall~u\in J'\Big\}\\
 \hspace{1.1in} \text{if}&
{[| \varphi|]_{J'}}= \!\!\!{\ds \sup_{\{k\geq 0\,\,|\, J'+k\,\subseteq J\}}}\!\!\!  \inf_{u\in J'}
\big|\Delta \varphi(u;k)\big|,
\\
\noalign{\medskip}
\sup\Big\{k\leq 0\,\,|\, J'+k\,\subseteq J,& \big|\Delta \varphi(u;k)\big|>{[| \varphi|]_{J'}}-\varepsilon
\ \ \forall~u\in J'\Big\}\\
 \hspace{1.1in} \text{if}&
{[| \varphi|]_{J'}}= \!\!\!{\ds \sup_{\{k\leq 0\,\,|\, J'+k\,\subseteq J\}}}\!\!\!  \inf_{u\in J'}
\big|\Delta \varphi(u;k)\big|.
\end{cases}
\end{equation}
We will also use the notation $\text{Tot.\!Var.}\{ \varphi;\,J'\}$ for the total variation
of $\varphi\in \text{BV}(J)$ on an interval $J'\subseteq J$ (e.g. see~\cite{Fol}).
\medskip

\noindent
We make the following standing assumptions on the flux function $f$:
\vspace{-5pt}
\begin{enumerate}
\item[{\bf (H1)}] $f: I =(i_-, i_+)\rightarrow \mathbb{R}$ 
is a twice continuously differentiable map;
\item[{\bf (H2)}] one of the following three conditions holds:
\begin{itemize}
\item[(i)] $\ds\lim_{u\to i_\pm} |f'(u)|<+\infty$\ \ \ and \ \  \  $\ds\lim_{u\to \pm \alpha} |f''(u)|<+\infty$\,; ~~~~
\item[(ii)] $i_+=+\infty$, \ \ $\ds\lim_{u\to +\infty} |f'(u)|=+\infty$ \ \ \ and \ \  \ $\ds\lim_{u\to +\infty} \frac{|f'(u)|}{\ds{\sup_{z\in(i_-,\,u)}}|f''(z)|}=+\infty$\,;
\smallskip
\item[(iii)] $i_-=-\infty$, \ \ $\ds\lim_{u\to -\infty} |f'(u)|=+\infty$ \ \ \ and \ \ \ $\ds\lim_{u\to -\infty} \frac{|f'(u)|}{\ds{\sup_{z\in(u,\,i_+)}}|f''(z)|}=+\infty$\,.
\end{itemize}
\end{enumerate}
\smallskip
\begin{theorem}\label{thm:glob-classic-controllability-nonconvex-1} 
Let  $f$ be a flux satisfying the assumptions {\bf (H1), (H2)-(i)},
and assume that $[|f|]_{I_1'}>0$, $[|f|]_{I_2'}>0$ for {intervals} $I_1', I_2'\subseteq I$.
Then,  given  any $a<b$, 
for  every~$\overline u\in C^1([a,b])$ and $\psi\in C^1([a,b])$, 
with $\mathrm{Im}(\overline u)\subsetneq I'_1$, $\mathrm{Im}(\psi)\subsetneq I'_2$,
and such that
\begin{gather}
\label{bound-initial-terminal-data-1}
\|\overline u'\|_{C^0([a,b])}< \frac{[|f|]_{I'_1}}{(b-a)
\cdot\! \|f''\|_{C^0(I)}}\,,
\qquad \quad
\|\psi'\|_{C^0([a,b])}< \frac{[|f|]_{I'_2}}{(b-a)
\cdot\! \|f''\|_{C^0(I)}}\,,
\end{gather}
and for any
\begin{equation}\label{T*-def}
T> {T^*:=T^*_1+T^*_2},\qquad\qquad T^*_1\doteq \frac{(b-a)}{[|f|]_{I'_1}}\,,\quad T^*_2\doteq \frac{(b-a)}{[|f|]_{I'_2}}
\end{equation}
there exists 
$h\in C^0([0,T])$
so that the Cauchy problem~\eqref{eq:blaw}, \eqref{eq:datum} admits a classical solution
$u\in C^1([0,T]\times[a,b])$,  that satisfies~\eqref{terminal-datum}. 
\end{theorem}

\begin{remark}
Notice that $T_1^*$ is the controllability time needed to steer the initial date $\bar{u} $ to $0$ while $T_2^*$ is the controllability time needed to steer $0$ to the final state $\psi$ . The controllability time $T^*$ in~\eqref{T*-def} is in general much smaller than 
the  
boundary controllability time  $\overline T$ in~\eqref{bdry-control-time}.
In particular, we observe  that
  $T^*_2\approx \frac{1}{\sup_{u\in I} |f'(u)|}$,
whereas  $\overline T \approx \frac{1}{\inf_{u\in \mathrm{Im}(\psi)} |f'(u)|}$. Therefore, whenever the target state $\psi$ is close to a critical state, i.e. {$\overline T \approx \frac{1}{\inf_{u\in \mathrm{Im}(\psi)} |f'(u)|}\to 0$}, we have
$\overline T\to +\infty$, while this is not the case for~$T^*$. 
\end{remark}

\begin{theorem}\label{thm:glob-classic-controllability-nonconvex-2} 
Let  $f$ be a flux satisfying the assumptions {\bf (H1)} and\,  {\bf (H2)-(ii)} or {\bf (H2)-(iii)}.
Then,  given  any $a<b$, and 
$T>0$,
for  every~$\overline u\in C^1([a,b])$ and $\psi\in C^1([a,b])$, 
there exists 
$h\in C^0([0,T])$
so that the Cauchy problem~\eqref{eq:blaw}, \eqref{eq:datum} admits a classical solution
$u\in C^1([0,T]\times[a,b])$,  that satisfies~\eqref{terminal-datum}. 
\end{theorem}

\begin{remark}
Clearly the flux $f(u)=\frac{u^2}{2}$ satisifies the   assumptions {\bf (H1)} and\,  {\bf (H2)-(ii)}.
Thus, as a particular case, we recover from Theorem~\ref{thm:glob-classic-controllability-nonconvex-2}
the global controllability result established in~\cite{Chap} for the Burgers equation
(by a quite different proof).
\end{remark}

\begin{theorem}\label{thm:glob-classic-controllability-convex-1} 
Let $f$ be a convex  map satisfying  the assumptions {\bf (H1), (H2)-(i)},
and assume that $[|f|]_{I_1'}>0$, $[|f|]_{I_2'}>0$ for {intervals} $I_1', I_2'\subseteq I$.
Then,  given  any $a<b$, $\rho>0$,
and $T>T^*$, with $T^*\geq 0$ as in~\eqref{T*-def},
there exists
 $C_1>0$ depending on  $b-a,\, T,\, T^*$,\, $\arg\mathrm{sup}{[|f|]_{I'_i,\,c_{1}}}$, $i=1,2$\, ($c_1$ being
 a constant depending on $\rho, T-T^*$),
so that the following hold.
 For  every $\overline u\in C^1([a,b])$ and $\psi\in C^1([a,b])$, 
with $\mathrm{Im}(\overline u)\subsetneq I'_1$, $\mathrm{Im}(\psi)\subsetneq I'_2$,
such that
\begin{equation}
\label{bound-initial-terminal-data-2}
\sup_{x\in [a,b]}\!\lfloor \overline u'(x)\rfloor_- \leq \!\frac{[|f|]_{I'_1}}{(b-a)\cdot\! \|f''\|_{C^0(I)}}-\rho \,,
\quad
\quad\sup_{x\in [a,b]}\!\lfloor  \psi'(x)\rfloor_+ \leq \!\frac{[|f|]_{I'_2}}{(b-a)\cdot\! \|f''\|_{C^0(I)}}-\rho\,,
\end{equation}
there exists 
$h\in C^0([0,T])$, with 
\begin{equation}
\label{tot-bar-bound1}
 \|h\|_{C^0([0,T])}+\text{Tot.\!Var.}\{ h;\,[0,T]\}\leq C_1\cdot 
\Big(1+ \|\overline u\|_{C^0([a,b])}+ \|\psi\|_{C^0([a,b])}
\Big),
\end{equation}
so that the Cauchy problem~\eqref{eq:blaw}, \eqref{eq:datum} admits a classical solution
$u\in C^1([0,T]\times[a,b])$,  that satisfies~\eqref{terminal-datum} and 
\begin{multline}
\label{tot-bar-bound2}
 \|u(t,\cdot)\|_{C^0([a,b])}+\text{Tot.\!Var.}\{ u(t,\cdot);\,[a,b]\}\\
 \leq C_1\cdot 
\Big( \|\overline u\|_{C^0([a,b])}+ \|\psi\|_{C^0([a,b])}+\!\sup_{x\in [a,b]}\lfloor \overline u'(x)\rfloor_-+\!\sup_{x\in [a,b]}\lfloor  \psi'(x)\rfloor_+\!\Big)
\end{multline}
for all $t\in (0,T)$.
\end{theorem}

\begin{remark}
In the case  $f$ is a convex  map satisfying  the assumptions {\bf (H1), (H2)-(i)},
and \, $i_+=+\infty$,   $\ds\lim_{\rho\to 0}\arg\mathrm{sup}{[|f|]_{I',\rho}}=+\infty$
(or $i_-=-\infty$ \, and \, $\ds\lim_{\rho\to 0}\arg\mathrm{sup}{[|f|]_{I',\rho}}=-\infty$), the constants $C_1, c_1>0$ provided by Theorem~\ref{thm:glob-classic-controllability-convex-1} have the following property:
If either $T\to T^*$ or $\rho\to 0$,
then $c_1\to 0$ and $C_1\to +\infty$.
\end{remark}

\begin{remark}
\label{concave-convex-1}
If $f$ is a concave  map satisfying  the assumptions {\bf (H1), (H2)-(i)} and $[|f|]_{I_1'}>0$, $[|f|]_{I_2'}>0$ for $I_1', I_2'\subseteq I$,
then the same conclusions of Theorem~\ref{thm:glob-classic-controllability-convex-1} hold, replacing
$\lfloor \overline u'(x)\rfloor_-$ with $\lfloor \overline u'(x)\rfloor_+ $
and $\lfloor  \psi'(x)\rfloor_+$ with $\lfloor  \psi'(x)\rfloor_-$
in the inequalities~\eqref{bound-initial-terminal-data-2}, \eqref{tot-bar-bound2}.
\end{remark}
\smallskip
\begin{theorem}\label{thm:glob-classic-controllability-convex-2} 
Let $f$ be a convex  or concave map satisfying  the assumptions {\bf (H1)} and\,  {\bf (H2)-(ii)} or {\bf (H2)-(iii)}.
Then,  given  any $a<b$, and $T> 0$, 
 for  every $\overline u\in C^1([a,b])$ and $\psi\in C^1([a,b])$, 
there exists 
$h\in C^0([0,T])$ 
so that the Cauchy problem~\eqref{eq:blaw}, \eqref{eq:datum} admits a classical solution
$u\in C^1([0,T]\times[a,b])$,  that satisfies~\eqref{terminal-datum}.
\end{theorem}

\begin{remark}
If $f: I =(i_-, +\infty)\rightarrow \mathbb{R}$ is a map satisfying  the assumptions {\bf (H1)} and\,  {\bf (H2)-(ii)},
and $I'_1, I'_2\subset I$ are bounded intervals, then setting
\begin{equation}
\label{derf-[||]-def}
[|f|]_{I'\!,\, u}\doteq  \sup_{\{k\,\,|\, I'+k\,\subseteq (i_-,\,u)\}}\, \inf_{v\in I'}
\big|\Delta \textcolor{red}{f(v;k)}\big|\,,
\end{equation}
 one finds
\begin{equation}
\label{derf-[||]-asympt}
\lim_{u \to+\infty}\dfrac{\,[|f|]_{I'\!,\, u}\,}{|f'(u)|}=1\,.
\end{equation}
Hence, taking the limit as $u \to\infty$ 
in~\eqref{bound-initial-terminal-data-1}, \eqref{T*-def}, \eqref{bound-initial-terminal-data-2}
with $[|f|]_{I'_i,\, u}$ in place of $[|f|]_{I'_i}$, 
and 
$\|f''\|_{C^0((i_-,\,u))}$ in place of 
$\|f''\|_{C^0(I)}$, 
the controllability time $T^*$ in~\eqref{T*-def} results to be zero and the upper bounds in \eqref{bound-initial-terminal-data-1}, \eqref{bound-initial-terminal-data-2} becomes~$+\infty$.
Therefore,
 at least formally, one can deduce the conclusions of Theorem~\ref{thm:glob-classic-controllability-nonconvex-2}
from Theorem~\ref{thm:glob-classic-controllability-nonconvex-1}, and the conclusions of Theorem~\ref{thm:glob-classic-controllability-convex-2}
from Theorem~\ref{thm:glob-classic-controllability-convex-1}.
Similar formal deductions can be carried out in the case $f$ satisfies the assumptions {\bf (H1)} and\,  {\bf (H2)-(iii)}.
\end{remark}

Relying on Theorem~\ref{thm:glob-classic-controllability-convex-1} we then establish 
a global  controllability result for BV states that 
satisfy one-sided Lipschitz inequalities expressed in terms of Dini derivatives.
We recall that
\begin{equation}
D^- \omega(x)= \liminf_{h\rightarrow0}   \frac{\omega(x+h)-\omega(x)}{h} ,\quad \quad 
D^+ \omega(x)= \limsup_{h\rightarrow0}   \frac{\omega(x+h)-\omega(x)}{h},
\end{equation}
denote, respectively, the lower and
the  upper Dini derivative of a function $\omega$ at $x$.
\medskip

\begin{theorem}\label{thm:glob-controllability-convex-1} 
Under the same assumptions of Theorem~\ref{thm:glob-classic-controllability-convex-1},
given  any $a<b$, $\rho>0$, and $T>T^*$, with $T^*\geq 0$ as in~\eqref{T*-def},
there exists $C_2>0$ depending on  $b-a,\, T,\, T^*$,\, $\arg\mathrm{sup}{[|f|]_{I'_i,c_{2}}}$, $i=1,2$\, ($c_2$ being
 a constant depending on $\rho, T-T^*$),
so that the following hold.
 For  every $\overline u\in BV([a,b])$ and $\psi\in BV([a,b])$,  
with $\mathrm{Im}(\overline u)\subsetneq I'_1$, $\mathrm{Im}(\psi)\subsetneq I'_2$,
and such that
\begin{equation}
\label{cond-iniital-terminal-data}
\begin{cases}
d^-&\doteq~
\sup_{x\in[a,b]} \left\lfloor D^- \overline u(x)\right\rfloor_-<~\frac{[|f|]_{I'_1}}{(b-a)\cdot\! \|f''\|_{C^0(I)}}-\rho\,,
\\
d^+&\doteq~
\sup_{x\in[a,b]} \left\lfloor D^+\psi(x)\right\rfloor_+< \frac{[|f|]_{I'_2}}{(b-a)\cdot\! \|f''\|_{C^0(I)}}-\rho\,,
\end{cases}
\end{equation}
there exists 
$h\in BV([0,T])$, with 
\begin{equation}
\label{tot-bar-bound3}
\|h\|_{{\bf L}^\infty([0,T])}+
\text{Tot.\!Var.}\{ h;\,[0,T]\}\leq C_2\cdot 
\Big( 1+\|\overline u\|_{{\bf L}^{\infty}([a,b])}+ \|\psi\|_{{\bf L}^{\infty}([a,b])}
\Big),
\end{equation}
so that the Cauchy problem~\eqref{eq:blaw}, \eqref{eq:datum} admits an entropy weak solution
on $[0,T]\times[a,b])$,  that satisfies~\eqref{terminal-datum} and 
\begin{equation}
\label{tot-bar-bound4}
\|u(t,\cdot)\|_{{\bf L}^\infty([a,b])}+
\text{Tot.\!Var.}\{ u(t,\cdot);\,[a,b]\}\leq C_2\cdot 
\Big( \|\overline u\|_{{\bf L}^{\infty}([a,b])}+ \|\psi\|_{{\bf L}^{\infty}([a,b])}
+d^-+d^+\Big)
\end{equation}
for all $t\in (0,T)$
\end{theorem}

\begin{theorem}\label{thm:glob-controllability-convex-1b} 
Let $f$ be a convex   map satisfying  the assumptions {\bf (H1)} and\,  {\bf (H2)-(ii)} or {\bf (H2)-(iii)}.
Then, given  any $a<b$, and $T> 0$,
for  every $\overline u\in BV([a,b])$ and $\psi\in BV([a,b])$, with
\begin{equation}
\label{cond-iniital-terminal-data-22}
d^-\doteq \sup_{x\in[a,b]} \left\lfloor D^- \overline u(x)\right\rfloor_-<+\infty\,,
\qquad\quad d^+\doteq \sup_{x\in[a,b]} \left\lfloor D^+\psi(x)\right\rfloor_+<+\infty\,,
\end{equation}
there exists $h\in BV([0,T])$ 
so that the Cauchy problem~\eqref{eq:blaw}, \eqref{eq:datum} admits an entropy weak solution
on $[0,T]\times[a,b])$,  that satisfies~\eqref{terminal-datum}.
\end{theorem}

\begin{remark}
By the proofs of Theorem~\ref{thm:glob-controllability-convex-1}
it follows that, in its setting, we obtain an approximate controllability result for classical solutions.
Namely, if $\overline u$, $\psi$ are BV states that satisfy conditons~
\eqref{cond-iniital-terminal-data},
then for any $T>T^*$,
and for every fixed $\varepsilon>0$, there exist $h\in C^0([0,T])$ and
a classical solution $u\in C^1([0,T]\times[a,b])$ of~\eqref{eq:blaw} that satisfies
\begin{equation}
\label{approx-controll}
\big\|u(0,\cdot)- \overline u\big\|_{{\bf L^1}([a,b])}<\varepsilon\,,\qquad\quad
\big\|u(T,\cdot)- \psi\big\|_{{\bf L^1}([a,b])}<\varepsilon\,.
\end{equation}
On the other hand, in the setting of Theorem~\ref{thm:glob-controllability-convex-1b}
the same type of approximate controllability result holds for any pair of initial and terminal data 
$\overline u\in BV([a,b]), \psi\in BV([a,b])$. In fact, in this case we can approximate $\overline u, \psi$ with $\overline u_\varepsilon$, $\psi_\varepsilon \in C^1([a,b])$
so that $\|\overline u-\overline u_\varepsilon\|_{\bf L^1}<\varepsilon$,\, $\|\psi-\psi_\varepsilon\|_{\bf L^1}<\varepsilon$.
Then, for any $T>0$, 
and for every fixed $\varepsilon>0$,
applying Thorem~\ref{thm:glob-classic-controllability-convex-2} we deduce the existence
of $h\in C^0([0,T])$ and of
a classical solution $u\in C^1([0,T]\times[a,b])$ of~\eqref{eq:blaw} that satisfies~\eqref{approx-controll}.
\end{remark}

\begin{remark}
\label{concave-convex-2}
If $f$ is a concave  map satisfying  the assumptions {\bf (H1), (H2)-(i)} and $[|f|]_{I'}>0$ for $I'\subseteq I$,
or  {\bf (H1)} and\,  {\bf (H2)-(ii)},  or {\bf (H1)} and\, {\bf (H2)-(iii)},
then the same conclusions of Theorem~\ref{thm:glob-controllability-convex-1} 
and of Theorem~\ref{thm:glob-controllability-convex-1b} hold, replacing
$\lfloor D^- \overline u(x)\rfloor_-$ with $\lfloor D^+ \overline u(x)\rfloor_+$
and $\lfloor D^+\psi(x)\rfloor_+$ with $\lfloor D^-\psi(x)\rfloor_-$
in the inequalities~\eqref{bound-initial-terminal-data-2}, \eqref{tot-bar-bound2}.
\end{remark}

\begin{remark}
Theorem~\ref{thm:glob-controllability-convex-1b} shows that, for conservation laws with convex or concave fluxes satisfying the assumptions
 {\bf (H1)} and\,  {\bf (H2)-(ii)} or {\bf (H2)-(iii)},
by choosing a suitable source term $h$ in~\eqref{eq:blaw}
we can steer in any arbitrarly small time $T>0$, every initial BV state~$\overline u$ which does not admit shock discontinuities to every BV target state $\psi$
which does not admit discontinuities generating a rarefaction wave.
This result is included in the ones established in~\cite{Perrollaz}, but here we obtain the solution u
as limit of smooth solutions, which are easier to handle both for numeric schemes and for treating 
similar problems in the case of diagonal systems of conservation laws. 
\end{remark}

The general strategy adopted  in Section~\ref{sec:contr-class} to establish the main results of the paper is basically an application of the 
so-called {\it return method}
introduced by Coron (see~\cite{coron})
in combination with the analysis of the Riccati equation governing the evolution of the space derivative
of the solutions. In fact, exploiting the a-priori bounds on the solutions of the Riccati equation, we
construct a source control which steers in a minimal time  the initial data $\overline u$ to some constant state, say $w_1$,
that can be quite far from the initial and terminal states $\overline u$, $\psi$. Similarly, one can produce a
source control that steers in minimal time some constant states, say $w_2$ (far away from $\overline u$, $\psi$),
to the terminal state $\psi$. Then, it's straightforward to see that we can connect $w_1$ and $w_2$ in
an arbitarly small time $\tau$, taking $h$ so that $\ds\int_0^{\tau} h(t) dt = w_2-w_1$. 
In the case of convex flux $f$, the explicit construction of the source control allows to provide a-priori estimates on the control and on the
solution  of~\eqref{eq:blaw} in terms of the ${\bf L^\infty}$ norm of $\overline u$, $\psi$,  of the negative variation of $\overline u$,
and of the positive variation of $\psi$. In turn, such a-priori bounds are crucial to establish in Section~\ref{sec:contr-entropy-weak} the corresponding
controllability results in the $BV$-setting.
Some exemplifying applications for traffic flow and sedimentation models 
are illustrated in Section~\ref{sec:appl}.

\section{Global controllability of $C^1$ states}
\label{sec:contr-class}

\subsection{Reduction to null controllability}

Since classical solutions of~\eqref{eq:blaw} are time reversible, 
 we can recover  
 the global controllability of $C^1$ states provided by Theorems~\ref{thm:glob-classic-controllability-nonconvex-1}-\ref{thm:glob-classic-controllability-nonconvex-2}-\ref{thm:glob-classic-controllability-convex-1}-\ref{thm:glob-classic-controllability-convex-2}
  from the null controllability of~\eqref{eq:blaw}.
Thus, it will be sufficient to prove:

\begin{proposition}
\label{C1-null-controll-nonconvex-1}
In the same setting and with the same assumptions of Theorem~\ref{thm:glob-classic-controllability-nonconvex-1},
for any $T>T^*_1$ with $T^*_1\geq 0$  as in~\eqref{T*-def}, and
 for  every~$\overline u\in C^1([a,b])$ with $\mathrm{Im}(\overline u)\subsetneq I'_1$,
and satisfying
\begin{equation}
\label{bound-initial-terminal-data-11}
\|\overline u'\|_{C^0([a,b])}< \!\frac{[|f|]_{I'_1}}{(b-a)
\cdot\! \|f''\|_{C^0(I)}}\,,
\end{equation}
there exists 
$h\in C^0([0,T])$ vanishing at $t=0,T$,
so that the Cauchy problem~\eqref{eq:blaw}, \eqref{eq:datum} admits a classical solution
$u\in C^1([0,T]\times[a,b])$  that satisfies 
\begin{equation}
\label{zero-terminal-datum}
u(T,x)=0\qquad x\in [a,b]\,.
\end{equation}
\end{proposition}

\begin{proposition}
\label{C1-null-controll-nonconvex-2}
In the same setting and with the same assumptions of Theorem~\ref{thm:glob-classic-controllability-nonconvex-2},
for any $T> 0$, and
 for  every~$\overline u\in C^1([a,b])$, 
there exists 
$h\in C^0([0,T])$ vanishing at $t=0,T$,
so that the Cauchy problem~\eqref{eq:blaw}, \eqref{eq:datum} admits a classical solution
$u\in C^1([0,T]\times[a,b])$  that satisfies~\eqref{zero-terminal-datum}. 
\end{proposition}

\begin{proposition}
\label{C1-null-controll-convex-1}
In the same setting and with the same assumptions of Theorem~\ref{thm:glob-classic-controllability-convex-1},
given  any $a<b$,  $\rho>0$,
and $T>T^*_1$, with $T^*_1\geq 0$ as in~\eqref{T*-def},
there exists $C_1>0$ depending on  $b-a,\, T,\, T^*_1$,\, $\arg\text{sup}{[|f|]_{I'_1,c_{1}}}$\, ($c_{1}$ being
 a constant depending on $\rho, T-T^*_1$),
so that the following hold.
  For  every $\overline u\in C^1([a,b])$, with $\mathrm{Im}(\overline u)\subsetneq I'_1$,
and satisfying
  \begin{equation}
\label{bound-initial-terminal-data-22}
\sup_{x\in [a,b]}\lfloor \overline u'(x)\rfloor_- 
\leq \!\frac{[|f|]_{I'_1}}{(b-a)\cdot\! \|f''\|_{C^0(I)}}-\rho\,,
\end{equation}
  there exists 
$h\in C^0([0,T])$ vanishing at $t=0,T$,
with 
\begin{equation}
\label{tot-bar-bound11}
 \|h\|_{C^0([0,T])}+\text{Tot.\!Var.}\{ h;\,[0,T]\}\leq C_1\cdot 
\Big(1+ \|\overline u\|_{C^0([a,b])}
\Big),
\end{equation}
so that the Cauchy problem~\eqref{eq:blaw}, \eqref{eq:datum} admits a classical solution
$u\in C^1([0,T]\times[a,b])$ that satisfies~\eqref{zero-terminal-datum} and 
\begin{equation}
\label{tot-bar-bound22}
 \|u(t,\cdot)\|_{C^0([a,b])}+\text{Tot.\!Var.}\{ u(t,\cdot);\,[a,b]\}\leq C_1\cdot 
\Big( \|\overline u\|_{C^0([a,b])}+\sup_{x\in [a,b]}\lfloor \overline u'(x)\rfloor_-\Big)
\qquad\forall~t\in (0,T)\,.
\end{equation}
\end{proposition}

\begin{proposition}
\label{C1-null-controll-convex-2}
In the same setting and with the same assumptions of Theorem~\ref{thm:glob-classic-controllability-convex-2},
given  any $a<b$ and $T> 0$,
for  every $\overline u\in C^1([a,b])$    there exists 
$h\in C^0([0,T])$ vanishing at $t=0,T$,
so that the Cauchy problem~\eqref{eq:blaw}, \eqref{eq:datum} admits a classical solution
$u\in C^1([0,T]\times[a,b])$ that satisfies~\eqref{zero-terminal-datum}.
\end{proposition}

The following lemmas show that
Theorems~\ref{thm:glob-classic-controllability-nonconvex-1}-\ref{thm:glob-classic-controllability-nonconvex-2}-\ref{thm:glob-classic-controllability-convex-1}-\ref{thm:glob-classic-controllability-convex-2}
are indeed a consequence of Proposition~\ref{C1-null-controll-nonconvex-1}-\ref{C1-null-controll-nonconvex-2}-\ref{C1-null-controll-convex-1}-\ref{C1-null-controll-convex-2}.
\smallskip

\begin{lemma}
\label{first-implication}
Proposition~\ref{C1-null-controll-nonconvex-1} \ $\Longrightarrow$ \ Theorem~\ref{thm:glob-classic-controllability-nonconvex-1}\,,
\  \ Proposition~\ref{C1-null-controll-nonconvex-2} \ $\Longrightarrow$ \ Theorem~\ref{thm:glob-classic-controllability-nonconvex-2}\,.
\end{lemma}
\vspace{-15pt}
\begin{proof}
We provide only a proof of the first implication, the second being entirely similar.
Let $T>T^*$ and, given  $\overline u\in C^1([a,b])$, $\psi\in C^1([a,b])$,
with $\mathrm{Im}(\overline u)\subsetneq I'_1$, $\mathrm{Im}(\psi)\subsetneq I'_2$,
which satisfy~\eqref{bound-initial-terminal-data-1},
set
\begin{equation}
\label{initial-data-def1}
\overline u_1(x)\doteq \overline u(x)\qquad\mathrm{and}\qquad \overline u_2(x)\doteq \psi(a+b-x)\qquad x\in[a,b]\,.
\end{equation}
Observe that $\overline u_1, \overline u_2$ satisfy the assumptions~\eqref{bound-initial-terminal-data-11}
(with $I'_2$ in place of $I'_1$ for $\overline u_2$).
Hence,  by Theorem~\ref{C1-null-controll-nonconvex-1} 
there exist $h_i\in C^0([0,T_i])$, $T_i>T^*_i$, $i=1,2$, 
vanishing at $t=0,T_i$,
and $u_i\in C^1([0,T_i]\times[a,b])$,  $i=1,2$, with $T=T_1+T_2$, that satisfiy 
\begin{equation}\label{eq:blaw-i}
\begin{aligned}
&\partial_t u_i+\partial_x f(u_i)=h_i(t)\,,~~~~t\in [0,T_i],\ \ \ x\in [a,b]\,,
\\ 
\noalign{\smallskip}
&u_i(0,x)=\overline u_i(x),\qquad {u_i(T_i,x)}=0,\ \quad x\in [a,b]\,.
\end{aligned}
\end{equation}
Consider the function 
\begin{equation}
\label{u-def-1}
u(t,x)=
\begin{cases}
u_1(t,x)\quad &\text{if}\qquad t\in[0,T_1],\ \ x\in[a,b],
\\
\noalign{\smallskip}
u_2(T-t,a+b-x)\quad &\text{if}\qquad t\in[T_1,T],\ \ x\in[a,b],
\end{cases}
\end{equation}
and define 
\begin{equation}
\label{h-def-1}
h(t)=
\begin{cases}
h_1(t)\quad &\text{if}\qquad t\in[0,T_1],
\\
\noalign{\smallskip}
-h_2(T-t)\quad &\text{if}\qquad t\in[T_1,T].
\end{cases}
\end{equation}
Then, relying on~\eqref{eq:blaw-i}, by a direct computation it follows that 
$u(t,x)$ is a solution of~\eqref{eq:blaw}. Moreover, since~\eqref{eq:blaw-i}
together with $h_1(T_1)=h_2(T_2)=0$
imply that $u_1(T_1,\cdot)=u_2(T_2,\cdot)=\partial_t u_1(T_1,\cdot)=\partial_t u_2(T_2,\cdot)\equiv 0$, 
we deduce that $u$ is a continuously differentiable map on $[0,T]\times [a,b]$.
Finally, observe that~\eqref{initial-data-def1}, \eqref{eq:blaw-i}, \eqref{u-def-1} yield $u(0,\cdot)=\overline u_1=\overline u$, $u(T,\cdot)=\overline u_2(a+b-\cdot)=\psi$,
which shows that $u$ is a $C^1$ classical solution of~\eqref{eq:blaw} steering the equation from $\overline u$
to~$\psi$.
\end{proof}


\begin{lemma}
Proposition~\ref{C1-null-controll-convex-1}  \ $\Longrightarrow$ \ Theorem~\ref{thm:glob-classic-controllability-convex-1}\,,
\ \ Proposition~\ref{C1-null-controll-convex-2} \ $\Longrightarrow$ \ Theorem~\ref{thm:glob-classic-controllability-convex-2}\,.
\end{lemma}
\vspace{-15pt}
\begin{proof}
We provide only a proof of the first implication, the second being entirely similar.
Let $T>T^*$ and, given  $\overline u\in C^1([a,b])$, $\psi\in C^1([a,b])$, 
with $\mathrm{Im}(\overline u)\subsetneq I'_1$, $\mathrm{Im}(\psi)\subsetneq I'_2$,
which satisfy~\eqref{bound-initial-terminal-data-2},
adopting the same setting~\eqref{initial-data-def1}
we observe that
\begin{equation}
\label{initial-data-bound-2}
\begin{aligned}
&\ \  \|\overline u_1\|_{C^0([a,b])}= \|\overline u\|_{C^0([a,b])},\quad
  \|\overline u_2\|_{C^0([a,b])}= \|\psi\|_{C^0([a,b])},\quad \sup_{x\in [a,b]}\lfloor \overline u'_1(x)\rfloor_- =\sup_{x\in [a,b]}\lfloor \overline u'(x)\rfloor_-,
  \\
  \noalign{\medskip}
  &\qquad\quad\mathrm{and}\qquad 
  \sup_{x\in [a,b]}\lfloor \overline u'_2(x)\rfloor_- =\sup_{x\in [a,b]}\lfloor -\psi'(a+b-x)\rfloor_-
  =\sup_{x\in [a,b]}\lfloor \psi'(x)\rfloor_+\,.
\end{aligned}
\end{equation}
Hence, relying on Proposition~\ref{C1-null-controll-convex-1} and
following the same  arguments of the proof of Lemma~\ref{first-implication}
 we deduce that the function $u$ defined in~\eqref{u-def-1}
 is a $C^1$ classical solution of~\eqref{eq:blaw}, with $h$ as in~\eqref{h-def-1}, steering the equation from $\overline u$
to~$\psi$.
Moreover, by Theorem~\ref{C1-null-controll-convex-1} we are assuming that
\begin{equation}
\label{tot-bar-bound12}
 \|h_i\|_{C^0([0,T])}+\text{Tot.\!Var.}\{ h_i;\,[0,T]\}\leq C_1\cdot 
\Big( 1+\|\overline u_i\|_{C^0([a,b])}
\Big),
\quad i=1,2,
\end{equation}
and
\begin{equation}
\label{tot-bar-bound23}
 \|u_i(t,\cdot)\|_{C^0([a,b])}+\text{Tot.\!Var.}\{ u_i(t,\cdot);\,[a,b]\}\leq C_1\cdot 
\Big( \|\overline u_i\|_{C^0([a,b])}+\sup_{x\in [a,b]}\lfloor \overline u'_i(x)\rfloor_-\Big)
\end{equation}
for all $t\in (0,T_i)$, $i=1,2$. Observe that, by~\eqref{u-def-1}-\eqref{h-def-1}, there holds
\begin{equation}
\begin{gathered}
\label{u-h-linf-tv-bound-1}
 \|h\|_{C^0([0,T])}\leq \max_i  \|h_i\|_{C^0([0,T_i])},\qquad\quad
 \text{Tot.\!Var.}\{ h;\,[0,T]\}\leq \sum_i \text{Tot.\!Var.}\{ h_i;\,[0,T_i]\},
  \\
  \noalign{\smallskip}
    \|u(t,\cdot)\|_{C^0([a,b])}\leq
    \begin{cases}
     \|u_1(t,\cdot)\|_{C^0([a,b])}\quad&\text{if}\quad t\in [0,T_1],
     \\
     \|u_2(T-t,\cdot)\|_{C^0([a,b])}\quad&\text{if}\quad t\in [T_1,T],
     \end{cases}
     \\
  \noalign{\smallskip}
  \text{Tot.\!Var.}\{ u(t,\cdot);\,[a,b]\}\leq 
   \begin{cases}
   \text{Tot.\!Var.}\{ u_1(t,\cdot);\,[a,b]\}\quad&\text{if}\quad t\in [0,T_1],
  \\
  \text{Tot.\!Var.}\{ u_2(T-t,\cdot);\,[a,b]\}\quad&\text{if}\quad t\in [T_1,T].
   \end{cases}
  \end{gathered}
\end{equation}
Thus, from~\eqref{tot-bar-bound12}-\eqref{tot-bar-bound23} we deduce that the functions $h$, $u$
defined in~\eqref{h-def-1}, \eqref{u-def-1}, respectively, satisfy the bounds~\eqref{tot-bar-bound1}-\eqref{tot-bar-bound2}
stated in Theorem~\ref{thm:glob-classic-controllability-convex-1} (with a costant $C_1$ different from the one
provided by Proposition~\ref{C1-null-controll-convex-1}).
\end{proof}

\subsection{Null controllability}

\noindent
{\bf Proof of Proposition~\ref{C1-null-controll-nonconvex-1}.}\\
\vspace{-10pt}

\noindent
{\bf 1.} 
%
Given $T>T^*_1$ with $T^*_1\geq 0$  as in~\eqref{T*-def}, 
and $\overline u\in C^1([a,b])$ with $\mathrm{Im}(\overline u)\subsetneq I'_1$,
 satisfying~\eqref{bound-initial-terminal-data-11}, 
  let $\varepsilon_1>0$ be such that
\begin{equation}
\label{T_0-def}
T>T_0\doteq \frac{(b-a)}{[|f|]_{I'_1}}\cdot\left(1+ 
2\,\varepsilon_1\right),
\end{equation}
\begin{equation}
\label{bound-initial-data-111}
\|\overline u'\|_{C^0([a,b])}< 
\!\frac{[|f|]_{I'_1}}{(b-a)\!\cdot\!\left(1+ 
3\,\varepsilon_1\right)\!\cdot\! \|f''\|_{C^0(I)}}\,.
\end{equation}
Then, 
we extend $\overline u$ to a continuously differentiable
function on the entire line $\R$, 
 that we still denote $\overline u$, so that 
 \begin{gather}
 \label{new-initial-data-1}
 \mathrm{Im}(\overline u)\subsetneq I'_1,\qquad\qquad
 \|\overline u'\|_{C^0(\R)}<\!\frac{[|f|]_{I'_1}}{(b-a)\!\cdot\!\left(1+ 
3\,\varepsilon_1\right)\!\cdot\! \|f''\|_{C^0(I)}}\,,
 %
 \\
  \noalign{\medskip}
   \label{new-initial-data-3}
   \overline u(x)=\begin{cases}
   \alpha_-\quad&\text{if}\qquad x\leq a-\varepsilon_1\cdot (b-a)\,,
   \\
   \noalign{\smallskip}
    \alpha_+\quad&\text{if}\qquad x\geq b+\varepsilon_1\cdot (b-a)\,,
   \end{cases}
 \end{gather}
for some constants $\alpha_-, \alpha_+\in\R$.
%
\begin{figure}
\begin{center}
\hspace{-2truecm}\resizebox{!}{4.5truecm}{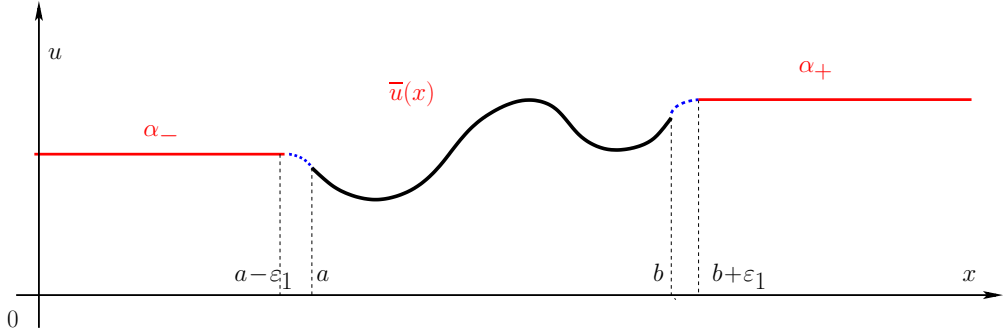}
\caption{Extension of the initial data}
\end{center}
\end{figure}
\medskip

Observe that, for any  $h\in C^0([0,+\infty])$,  the Cauchy problem 
\begin{equation}
\label{eq:cpblawonr}
\begin{aligned}
&\partial_t u+\partial_x f(u)=h(t)\,,~~~~t>0,\ x\in\R\,,
\\
\noalign{\medskip}
&\qquad 
u(0,x)~=~\overline{u}(x)\,,~~~~x\in \R\,,
\end{aligned}
\end{equation}
admits a classical solution $u(t,x)$ defined on some maximal interval $\big[0, T^h\big)$.
Given any fixed $x_0\in\R$, let $x(\cdot)$ denote the unique 
forward characteristics of~\eqref{eq:cpblawonr} starting from $x_0$, i.e.
the unique solution of
\begin{equation}
\label{char-eq}
\dot{x}(t)~=~f'(u(t ,x(t))), \quad t\in \big[0, T^h\big),
\end{equation}
satisfying $ x(0)=x_0$.
Then, $z_0(t)\doteq u(t,x(t))$ is a {Carath\'eodory}
solution of 
\[
\dot{z_0}(t)~=~h(t), \quad t\in [0, T^h),\qquad z_0(0)~=~\overline u(x_0)\,.
\]
On the other hand, observe that the function $w(t,x)=\partial_x u(t,x)$ is a broad solution on $[0,T^h\,) \times\R$ of the semilinear equation
\begin{equation}
\label{semil-eq1}
\partial_t w(t,x)+f'(u(t,x))\cdot \partial_x w(t,x)~=~-f''(u(t,x))\cdot w^2(t,x)\,.
\end{equation}
Hence, relying on~\eqref{semil-eq1} we deduce that $z_1(t)\doteq \partial_x u(t,x(t))$ is a {Carath\'eodory}
solution of 
\begin{equation}
\label{ricc-eq}
\dot{z_1}(t)~=~-f''(u(t ,x(t)))\cdot z_1^2(t), \quad t\in [0, T^h\big),\qquad z_1(0)~=~\overline u'(x_0)\,.
\end{equation}
Then, a direct computation yields
\bel{u(t,x(t))}
z_0(t)~=~\overline{u}(x_0)+\int_{0}^th(\tau)d\tau,
\eeq
\bel{x(t)}
x(t)~=~x_0+\int_{0}^{t}f'(z_0(\tau))d\tau~=~x_0+\int_{0}^{t}f'\left(\overline{u}(x_0)+\int_{0}^{\tau}h(s)ds\right)d\tau,
\eeq
and 
\bel{u_x(t,x(t))}
{1\over z_1(t)}~=~{1\over \overline{u}'(x_0)}+\int_{0}^{t}f''(z_0(\tau))d\tau\,,
\eeq
for all $t\in [0, T^h)$.
\quad\\

{\bf 2.} Consider the continuous function 
{
\begin{equation}
\begin{aligned}
\label{def-h-1}
h(t)~&=~\dfrac{t\cdot \overline{h}}{\tau_1}\cdot \chi_{[0,\tau_1]}+ \overline{h}\cdot \chi_{[\tau_1,T_0]}+\dfrac{(T_1-t)\cdot \overline{h}}{\tau_1}\cdot \chi_{[T_0,T_1]}\\
\noalign{\smallskip}
&\quad -\dfrac{16  (t-T_1)\cdot (\alpha+T_0\cdot\overline h\,)}{3\,(T-T_1)^2}\cdot \chi_{\Big[T_1,\, 
\dfrac{T+3T_1}{4}\Big]}-\dfrac{4\,(\alpha+T_0\cdot\overline h\,)}{3\,(T-T_1)}\cdot\chi_{\Big[\dfrac{T+3T_1}{4},\,
\dfrac{3T+T_1}{4}\Big]}\\
\noalign{\smallskip}
&\quad +\left[-\dfrac{4\,(\alpha+T_0\cdot\overline h\,)}{3\,(T-T_1)}\!+\! \dfrac{
4\,(4t-3T-T_1)
\!\cdot\! (\alpha+T_0\cdot\overline h\,)}{3\,(T-T_1)^2}\right]\cdot \chi_{\Big[\dfrac{3T+T_1}{4},\,
T\Big]}\,,
\end{aligned}
\end{equation}
}
%
\vspace{-7pt}

\noindent
where $\chi_{_{ J}}$ denotes the characteristic function of an interval $J\subset \R$, $T_0$ is the time defined in~\eqref{T_0-def}, 
while $\alpha\in\{\alpha_-,\,\alpha_+\}$,
and the constants $0<\tau_1<T-T_0$, $\overline h\in\R$, will be chosen later 
so that 
\begin{equation}
\label{time-exist-classical}
T_1\doteq T_0+\tau_1<T^h
\end{equation}
\big($T^h$~being the maximal time of existence of a classical solution to~\eqref{eq:cpblawonr}\big), and such that there holds 
\begin{equation}
\label{constant-terminal-datum}
u(T_1,x)=\alpha+ T_0\cdot \overline h\qquad x\in [a,b]\,.
\end{equation}
Notice that the definition of~\eqref{def-h-1}, together with~\eqref{constant-terminal-datum}, then implies
\begin{equation}
u(t,x)=\alpha+ T_0\cdot \overline h+\int_{T_1}^t h(s)~ds\qquad t\in [T_1,\, T]\,,\ \ \ x\in [a,b]\,, 
\end{equation}
which in turn, by a direct computation, yields
\begin{equation}
\label{zero-terminal-datum-222}
u(T,x)=\alpha+T_0\cdot \overline h-\alpha-T_0\cdot \overline h =0\qquad x\in [a,b]\,,
\end{equation}
thus showing that condition~\eqref{zero-terminal-datum} is verified.
\begin{figure}
\begin{center}
\hspace{-2truecm}\resizebox{!}{5truecm}{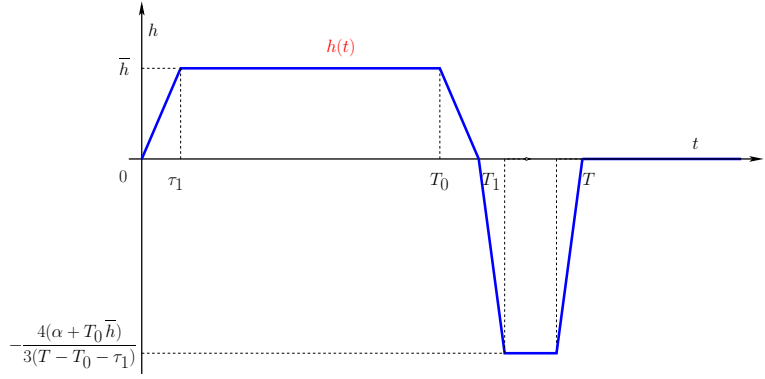}
\caption{The source control}
\end{center}
\end{figure}
Hence, in order to conclude the proof of the theorem we need only to establish~\eqref{time-exist-classical}-\eqref{constant-terminal-datum},
with $\alpha=\alpha_-$ or $\alpha=\alpha_+$.
To this end, relying on~\eqref{u(t,x(t))}-\eqref{u_x(t,x(t))}, we find
\begin{equation}\label{z0(t)}
z_0(t)~=~\overline{u}(x_0)+\ds{t^2\cdot \overline{h}\over 2\tau_1}\cdot \chi_{[0,\tau_1]}+\left(t-{\tau_1\over 2}\right)\cdot \bar{h}\cdot \chi_{[\tau_1,T_0]}+\left(T_0-\dfrac{(T_1-t)^2}{2\tau_1}\right)\cdot\overline h \cdot \chi_{[T_0,T_1]}
\end{equation}
for all $t\in [0,T_1]$ and 
\begin{equation}
\label{x(t)-2}
\begin{aligned}
x(T_1)&=x_0+\ds\int_{0}^{\tau_1}f'(z_0(s))ds+\int^{T_0}_{\tau_1}f'\Big(\overline{u}(x_0)+\ds{\tau_1 \cdot\overline{h}\over 2}+(s-\tau_1)\cdot \overline{h}\,\Big)~ds+
\int^{T_1}_{T_0}f'(z_0(s))ds\\
\noalign{\medskip}
&=x_0+\ds\int_{0}^{\tau_1}f'(z_0(s))ds+{1\over \bar{h}}\cdot \left[f\Big(\overline{u}(x_0)+\ds\Big(T_0-{\tau_1\over 2}\Big) \cdot \overline{h}\Big)-f\Big(\overline{u}(x_0)+\ds{\tau_1 \bar{h}\over 2}\Big)\right]+\int^{T_1}_{T_0}f'(z_0(s))ds
\\
\noalign{\allowdisplaybreaks}
\noalign{\pagebreak}
&=x_0+\ds\int_{0}^{\tau_1}f'(z_0(s))ds+T_0\cdot \Delta f\big(\bar{u}(x_0);\,T_0\cdot\overline h\,\big)
\\
\noalign{\medskip}
&\qquad\quad\ +{\tau_1\over 2}\cdot \left[
\Delta f\left(\overline{u}(x_0)+T_0\cdot\bar{h}; -{\tau_1\bar{h}\over 2}\right)
-\Delta f\left(\overline{u}(x_0);{\tau_1\bar{h}\over 2}\right)\right]+\int^{T_1}_{T_0}f'(z_0(s))ds\,.
\end{aligned}
\end{equation}
\noindent
{\bf 3.} 
Since we are assuming that $[|f|]_{I'_1}>0$, and because $\mathrm{Im}(\overline u)$ is a closed interval, recalling definition~\eqref{norm1-[Delta-def} and~\eqref{new-initial-data-1} there will be some $\overline k$ such that, either
\begin{equation}
\label{kappabar-def}
\Delta f(u;\,\overline k\,)> [|f|]_{I'_1} -\dfrac{\varepsilon_1\cdot (b-a)}{2\,T_0}\qquad\forall~u\in \mathrm{Im}(\overline u)\,,
\end{equation}
or
\begin{equation}
\label{kappabar-def22}
\Delta f(u;\,\overline k\,)< -[|f|]_{I'_1} +\dfrac{\varepsilon_1\cdot (b-a)}{2\,T_0}\qquad\forall~u\in \mathrm{Im}(\overline u)\,,
\end{equation}
with 
$\varepsilon_1$ as in~\eqref{T_0-def}-\eqref{bound-initial-data-111}.
To fix the ideas, assume that~\eqref{kappabar-def} holds and that $\overline k>0$.
Then, choosing 
\begin{equation}
\label{accabar-def}
\overline h = \dfrac{\overline k}{T_0}, 
\end{equation}
we find
\begin{equation}
\label{deltaf-est2}
\Delta f\big(\overline u(x_0);\,T_0\cdot\overline h\,\big)>
[|f|]_{I'_1} -\dfrac{\varepsilon_1\cdot (b-a)}{2\, T_0}\,.
\end{equation}
Hence, 
if 
{$x(T_1)\in [a,b]$}, 
and we choose
\begin{equation}
\label{tau1-cond1}
\tau_1<
\min\left\{
\dfrac{\varepsilon_1\cdot(b-a)}{6\cdot \|f'\|_{C^0(I)}}\,, T-T_0\right\},
\end{equation}
combining~\eqref{x(t)-2} with~\eqref{deltaf-est2}, and recalling~\eqref{T_0-def},  we derive 
\begin{equation}
\label{x(t)-3}
\begin{aligned}
x_0&\leq {x(T_1)} +3\tau_1 \cdot \|f'\|_{C^0(I)} -T_0\cdot \Delta f\big(\overline u(x_0);\,T_0\cdot\overline h\,\big)
\\
\noalign{\smallskip}
&\leq b - T_0\cdot [|f|]_{I'_1}  +3\tau_1 \cdot \|f'\|_{C^0(I)} + \dfrac{\varepsilon_1\cdot (b-a)}{2}
\\
\noalign{\smallskip}
&< b - T_0\cdot [|f|]_{I'_1}+ \varepsilon_1\cdot (b-a)=a-\varepsilon_1\cdot (b-a)\,.
\end{aligned}
\end{equation}
Because of~\eqref{new-initial-data-3}, \eqref{z0(t)}, the inequality~\eqref{x(t)-3} implies that ${u(T_1, x(T_1))}=
\overline u(x_0)+T_0\cdot\overline h=\alpha_-+T_0\cdot\overline h$,
which proves~\eqref{constant-terminal-datum}, 
choosing
\begin{equation}
\label{alpha-def-2}
\alpha=\alpha_-\,.
\end{equation}
On the other hand,  relying on~\eqref{T_0-def}, \eqref{new-initial-data-1}, \eqref{u_x(t,x(t))},
and taking
\begin{equation}
\label{tau1-cond2}
\tau_1<
\dfrac{\varepsilon_1\cdot(b-a)}{2\cdot [|f|]_{I'_1}}\,,
\end{equation}
we deduce that, if $\overline{u}'(x_0)\neq 0$, then
\begin{equation}
\label{z1(t)-est2}
\begin{aligned}
\dfrac{1}{|z_1(t)|}&\geq \dfrac{1}{|\overline{u}'(x_0)|}-\left|\int_{0}^{t}f''(z_0(\tau))d\tau\right|
\\
\noalign{\smallskip}
&\geq \dfrac{1}{ \|\overline u'\|_{C^0(\R)}}-t \cdot  \|f''\|_{C^0(I)}
\\
\noalign{\smallskip}
&>
\frac{(b-a)\!\cdot\!\left(1+ 3\,\varepsilon_1\right)\!\cdot\! \|f''\|_{C^0(I)}}{[|f|]_{I'_1}}-  T_1 \cdot  \|f''\|_{C^0(I)}
\\
\noalign{\smallskip}
&>\frac{(b-a)\cdot\varepsilon_1\cdot \|f''\|_{C^0(I)}}{2\cdot [|f|]_{I'_1}}\qquad\forall~t\in [0,{T_1}]\,.
\end{aligned}
\end{equation}
Therefore, choosing
\begin{equation}
\label{tau1-cond3}
\tau_1< \min\left\{\dfrac{\varepsilon_1\cdot(b-a)}{6\cdot \|f'\|_{C^0(I)}},\,
\dfrac{\varepsilon_1\cdot(b-a)}{[|f|]_{I'_1}},\, T-T_0
\right\},
\end{equation}
and observing that $z_1(t)\equiv 0$ \, if \, $\overline{u}'(x_0)= 0$,
we deduce from~\eqref{z1(t)-est2} that, for every solution $x(t)$ of~\eqref{char-eq}
starting at $x_0\in\R$, the function
$z_1(t)= \partial_x u(t,x(t))$ 
satisfies
\begin{equation}
|z_1(t)|<+\infty\qquad \forall~t\in [0,{T_1}]\,,
\end{equation}
 which yields~\eqref{time-exist-classical}.
This completes the proof of the theorem with the choice of $\overline h$, $\alpha$ 
and $\tau_1$ in~\eqref{def-h-1}
according with~\eqref{accabar-def}, \eqref{alpha-def-2}, \eqref{tau1-cond3}, respectively.

\qed

\noindent
{\bf Proof of Proposition~\ref{C1-null-controll-nonconvex-2}.}\\
\vspace{-10pt}

\noindent
To fix the ideas assume that the flux $f$ satisfies the assumptions {\bf (H1)} and\, {\bf (H2)-(ii)}.
Given $\overline u\in C^1([a,b])$, 
set $I'_1\doteq \mathrm{Im}(\overline u)$ and $I_u\doteq (i_-,\,u)$.
Observe that,  because of {\bf (H2)-(ii)}  and~\eqref{derf-[||]-asympt}, 
we have 
\begin{equation}
\nonumber
\lim_{u\to+\infty} \frac{(b-a)}{[|f|]_{I'_1,u}}=\lim_{u\to+\infty} \frac{(b-a)}{|f'(u)|}=0\,,
\end{equation}
and
\begin{equation}
\nonumber
\lim_{u\to+\infty}\!\frac{[|f|]_{I'_1,u}}{(b-a)
\cdot\! \|f''\|_{C^0(I_u)}}=\lim_{u\to+\infty}\!\frac{\|f'\|_{C^0(I_u)}}{(b-a)
\cdot\! \|f''\|_{C^0(I_u)}}=+\infty\,,
\end{equation}
where $[|f|]_{I'_1,u}$ is defined as in~\eqref{derf-[||]-def}.
Then, given any $T>0$, 
there will be $u_0>i_-$ such that
\begin{equation}
T> \frac{(b-a)}{[|f|]_{I'_1,{u_0}}}\,,\qquad\qquad
\|\overline u'\|_{C^0([a,b])}< \frac{[|f|]_{I'_1,u_0}}{(b-a)
\cdot\! \|f''\|_{C^0(I_{u_0})}}\,.
\end{equation}
Now, applying Theorem~\ref{C1-null-controll-nonconvex-1} to the flux $f: I_{u_0}\to \R$, 
and to the initial data $\overline u\in C^1([a,b])$,
which satisfy the assumptions {\bf (H1), (H2)-(i)}, $[|f|]_{I'_1,u_0}>0$,
and~\eqref{bound-initial-terminal-data-11},
respectively,  
we deduce the conclusion of Proposition~\ref{C1-null-controll-nonconvex-2}.
\qed
\bigskip

\noindent
{\bf Proof of Proposition~\ref{C1-null-controll-convex-1}.}\\
\vspace{-10pt}

\noindent
{\bf 1.} 
Given $\overline u\in C^1([a,b])$
 satisfying $\mathrm{Im}(\overline u)\subsetneq I'_1$ and~\eqref{bound-initial-terminal-data-22}, 
  let $\varepsilon_1>0$ (depending on $T-T^*_1$ and $\rho$)
  be such that $T>T_0$, with $T_0$ as in~\eqref{T_0-def}, and
\begin{equation}
\label{bound-initial-data-112}
 \!\frac{[|f|]_{I'_1}}{(b-a)\cdot\! \|f''\|_{C^0(I)}}-\rho<\!\frac{[|f|]_{I'_1}}{(b-a)\!\cdot\!\left(1+ 
3\,\varepsilon_1\right)\!\cdot\! \|f''\|_{C^0(I)}}\,.
\end{equation}
Then, in view of~\eqref{bound-initial-terminal-data-22}, \eqref{bound-initial-data-112}, 
we extend $\overline u$ to a continuously differentiable
function on $\R$, 
 that we still denote $\overline u$, so that 
 \begin{gather}
 \label{new-initial-data-4}
 \mathrm{Im}(\overline u)\subsetneq I'_1,\qquad\qquad
\sup_{x\in \R}\lfloor \overline u'(x)\rfloor_-<\!\frac{[|f|]_{I'_1}}{(b-a)\!\cdot\!\left(1+ 
3\,\varepsilon_1\right)\!\cdot\! \|f''\|_{C^0(I)}}\,,
\\
\noalign{\medskip}
 \label{new-initial-data-4b}
 \|\overline u\|_{C^0(\R)} \leq 2\cdot  \|\overline u\|_{C^0([a,b])}\,,\qquad\quad
\text{Tot.\!Var.}\{ \overline u;\,\R\}\leq 2\cdot \text{Tot.\!Var.}\{ \overline u;\,[a,b]\}\,,
 \\
  \noalign{\medskip}
   \label{new-initial-data-5}
   \overline u(x)=\begin{cases}
   \alpha_-\quad&\text{if}\qquad x\leq a-\varepsilon_1\cdot (b-a)\,,
   \\
   \noalign{\smallskip}
    \alpha_+\quad&\text{if}\qquad x\geq b+\varepsilon_1\cdot (b-a)\,,
   \end{cases}
 \end{gather}
for some constants 
\begin{equation}
\label{alphapm-def}
\alpha_-,\ \alpha_+\in\text{Im}(\overline u)\,.
\end{equation}
Next observe that, if we show that the Cauchy problem~\eqref{eq:cpblawonr},
with $h$ defined as in~\eqref{def-h-1}, admits a classical solution $u$ on $[0, {T_1}]\times\R$, with $T_0$ as in~\eqref{T_0-def},
and $\tau_1>0$ satisfying~\eqref{tau1-cond1}, then by the same arguments of the proof of  Theorem~\ref{C1-null-controll-nonconvex-1} we deduce that~\eqref{constant-terminal-datum}, \eqref{zero-terminal-datum-222} hold. Hence, in order to complete the proof 
that $u$ is a classical solution of~\eqref{eq:blaw}, \eqref{eq:datum}
 satisfying~\eqref{zero-terminal-datum}, it remains to prove
that~\eqref{time-exist-classical} is also true. To this end notice that, since $f''(u)$ is nonnegative (being $f$ a convex map),
by~\eqref{ricc-eq} it follows that $z_1$ is a decreasing map on $[0, T^h)$. 
Moreover, if $\overline u'(x_0)>0$ from~\eqref{u_x(t,x(t))} it follows that $z_1(t)>0$ for all $t\in [0,T^h)$.
On the other hand, in the case where $\overline u'(x_0)<0$, relying on~\eqref{T_0-def}, \eqref{u_x(t,x(t))}, \eqref{new-initial-data-4}, 
and taking $\tau_1$ as in~\eqref{tau1-cond2},
we deduce 
\begin{equation}
\label{z1(t)-est4}
\begin{aligned}
\dfrac{1}{z_1(t)}&\leq \dfrac{1}{\overline{u}'(x_0)}+\left|\int_{0}^{t}f''(z_0(\tau))d\tau\right|
\\
\noalign{\smallskip}
&\leq \dfrac{-1}{\sup_{x\in \R}\lfloor \overline u'(x)\rfloor_-}+t \cdot  \|f''\|_{C^0(I)}
\\
\noalign{\smallskip}
&<-\frac{(b-a)\!\cdot\!\left(1+ 3\,\varepsilon_1\right)\!\cdot\! \|f''\|_{C^0(I)}}{[|f|]_{I'_1}}+  {T_1} \cdot  \|f''\|_{C^0(I)}
\\
\noalign{\smallskip}
&<-\frac{(b-a)\cdot\varepsilon_1\cdot \|f''\|_{C^0(I)}}{2\cdot [|f|]_{I'_1}}
\qquad\forall~t\in [0, {T_1}]\,.
\end{aligned}
\end{equation}
Thus, choosing $\tau_1$ as in~\eqref{tau1-cond3}, we derive
\begin{equation}
-\infty< z_1(t)\leq \overline u'(x_0)\qquad\quad \forall~t\in {[0,T_1]}\,,\quad x_0\in\R\,,
\end{equation}
which shows that~\eqref{time-exist-classical} is verified.

\noindent
{\bf 2.} By the definition of $h$ in~\eqref{def-h-1}, and because of~\eqref{alphapm-def}, a direct computation yields
\begin{align}
\label{totvar-h-1}
\text{Tot.\!Var.}\{ h;\,[0,T]\}&=\frac{2\cdot |\overline k|}{T_0}+\frac{8\cdot|\alpha_\pm+\overline k|}{3\cdot(T-{T_1})}
\leq  \bigg(\frac{6\,T\!+\!2\,T_0\!-\!6\,\tau_1}{3\,T_0\cdot(T\!-{T_1})}\bigg)\!\cdot |\overline k|+\frac{8}{3(T\!-{T_1})}
\!\cdot \|\overline u\|_{C^0([a,b])}\,,
\\
\noalign{\bigskip} 
\label{c0-h-1}
\|h\|_{C^0([0, T])}&\leq \frac{\text{Tot.\!Var.}\{ h;\,[0,T]\}}{2}
\leq  \bigg(\frac{3\,T\!+\!T_0\!-\!3\,\tau_1}{3\,T_0\cdot(T\!-{T_1})}\bigg)\!\cdot |\overline k|+\frac{4}{3(T\!-{T_1})}
\!\cdot \|\overline u\|_{C^0([a,b])}\,,
\\
\noalign{\bigskip}
\label{inth-1}
&\qquad \left|\int_0^t h(s) ds\right|
\leq |\overline k|+|\alpha_\pm|
\leq  |\overline k|+\|\overline u\|_{C^0([a,b])}
\qquad\forall~t\in[0, T]\,,
\end{align}
where $\overline k=T_0\cdot\overline h$ is a constant choosen so that~\eqref{kappabar-def} holds
which, recalling~\eqref{argsup}, \eqref{T_0-def},  can be taken so that 
\begin{equation}
\label{barkappa-bound11}
|\overline k|\leq \arg\text{sup}{[|f|]_{I'_1,c_1}}+1,\qquad\quad
c_1\leq \frac{\varepsilon_1}{2(1+2\varepsilon_1)}\cdot [|f|]_{I'_1}.
\end{equation}
Then, choosing
\begin{equation}
\label{tau1-cond333}
\tau_1< \min\left\{\dfrac{\varepsilon_1\cdot(b-a)}{6\cdot \|f'\|_{C^0(I)}},\,
\dfrac{\varepsilon_1\cdot(b-a)}{[|f|]_{I'_1}},\, \frac{T-T_0}{2}
\right\},
\end{equation}
\eqref{totvar-h-1}, \eqref{c0-h-1} imply
\begin{equation}
\label{totvar-h-2}
 \|h\|_{C^0([0,T])}+\text{Tot.\!Var.}\{ h;\,[0,T]\}\leq \max\bigg\{
 \frac{6\,T\!+\!3\,T_0}{T_0\cdot(T\!-\!T_0)},\,
\frac{8}{T\!-\!T_0} 
\bigg\}
 \cdot \Big(|\overline k| + \|\overline u\|_{C^0([a,b])}\Big)\,,
\end{equation}
while from \eqref{u(t,x(t))}, \eqref{new-initial-data-4}, \eqref{inth-1}, we deduce 
\begin{equation}
\label{u-c0-est-234}
\|u(t,\,\cdot\,)\|_{C^0([a,b])}\leq |\overline k|+4\cdot \|\overline u\|_{C^0([a,b])}\,.
\end{equation}

Next, observe that, letting $ \text{Tot.\!Var.}\!^-\{ \overline u;\,[a,b]\}$ denote the negative variation of $ \overline u$ on $[a,b]$
(e.g. see~\cite{Fol}), one has
\begin{equation}
\label{totvar-est1}
\text{Tot.\!Var.}\{ \overline u;\,[a,b]\}\leq 2\, \big(\|\overline u\|_{C^0([a,b])}+\text{Tot.\!Var.}\!^-\{ \overline u;\,[a,b]\}\big)\,.
\end{equation}
Thus, 
we have 
\begin{equation}
\label{totvar-est2}
\text{Tot.\!Var.}\{ \overline u;\,[a,b]\}\leq 2\cdot(1+(b-a)) \cdot \big(\|\overline u\|_{C^0([a,b])}+
\sup_{x\in [a,b]}\lfloor \overline u'(x)\rfloor_-\big)\,.
\end{equation}
On the other hand, notice that a classical solution of~\eqref{eq:cpblawonr} is also the unique entropic weak solutions
of~\eqref{eq:cpblawonr}.
Hence,  since 
 scalar balance laws as in~\eqref{eq:cpblawonr}, with a source term $h$ only depending on time,
admit entropic weak solutions with total variation nonincreasing in time (e.g.  obtained by an operator splitting algorithm, see~\cite{DafermosBook}), 
relying also on~\eqref{new-initial-data-4b}  we derive
\begin{equation}
\label{totvar-est3}
\begin{aligned}
\text{Tot.\!Var.}\{ u(t,\cdot\,) ;\,[a,b]\}
&\leq \text{Tot.\!Var.}\{ u(t,\cdot\,) ;\,\R\}
\\
\noalign{\smallskip}
&\leq  \text{Tot.\!Var.}\{ \overline u ;\,\R\}
\\
\noalign{\smallskip}
&\leq  2\cdot \text{Tot.\!Var.}\{ \overline u ;\,[a,b]\}
\qquad\forall~t\in[0,T]\,.
\end{aligned}
\end{equation}
Then, combinig~\eqref{totvar-est2}, \eqref{totvar-est3}, we obtain
\begin{equation}
\label{totvar-est4}
\text{Tot.\!Var.}\{ u(t,\cdot\,) ;\,[a,b]\}\leq 4\cdot(1+(b-a)) \cdot \big(\|\overline u\|_{C^0([a,b])}+
\sup_{x\in [a,b]}\lfloor \overline u'(x)\rfloor_-\big)\,.
\end{equation}
Hence, \eqref{totvar-h-2}, \eqref{u-c0-est-234}, \eqref{totvar-est4}
show that the estimates~\eqref{tot-bar-bound11}, \eqref{tot-bar-bound22} are satisfied with
\begin{equation}
C_1= \max\bigg\{
\frac{(6\,T\!+\!3\,T_0)\cdot(1+|\overline k|)}{T_0\cdot(T\!-\!T_0)},\,
\frac{8\, (1+|\overline k|)}{T\!-\!T_0},
\, 4\,(2+(b-a))+|\overline k| \bigg\},
\end{equation}
where $\overline k$ satisfies the bound~\eqref{barkappa-bound11}.
This completes the proof of the theorem.
\qed

\bigskip

\noindent
{\bf Proof of Proposition~\ref{C1-null-controll-convex-2}.}\\
\vspace{-10pt}

\noindent
The conclusions of Proposition~\ref{C1-null-controll-convex-2} follow from Proposition~\ref{C1-null-controll-convex-1}
with the same arguments of the proof of Proposition~\ref{C1-null-controll-nonconvex-2}.
\qed
\smallskip

\pagebreak

\section{Controllability of BV states}
\label{sec:contr-entropy-weak}

\noindent
{\bf Proof of Theorem~\ref{thm:glob-controllability-convex-1}.}\\
\vspace{-10pt}

\noindent
Given $\overline u\in BV([a,b])$ and $\psi\in BV([a,b])$,  
with $\mathrm{Im}(\overline u)\subsetneq I'_1$, $\mathrm{Im}(\psi)\subsetneq I'_2$,
and such that~\eqref{cond-iniital-terminal-data} holds,
relying on Lemma~\ref{approx-bvmaps} in the Appendix there will be sequences
$\{\overline u_n\}_{n\geq 1}, \{\psi_n\}_{n\geq 1} \subset C^1([a,b])$, with
$\mathrm{Im}(\overline u_n)\subsetneq I'_1$, 
$\mathrm{Im}(\psi_n)\subsetneq I'_2$ 
such that
\begin{equation}
\label{initial-terminal-data-conv-1}
 \overline u_n\to \overline u\,,\qquad
 \psi_n \to \psi
  \qquad \text{in}\quad {\bf L^1} ([a,b])\,,
\end{equation}\
\vspace{-5pt}
and
\begin{equation}
\label{cond-iniital-terminal-data-222}
\sup_{x\in [a,b]}\!\lfloor \overline u'_n(x)\rfloor_- \leq ~\frac{[|f|]_{I'_1}}{(b-a)\cdot\! \|f''\|_{C^0(I)}}-\rho\,,
\qquad
\quad\sup_{x\in [a,b]}\!\lfloor  \psi'_n(x)\rfloor_+ \leq ~\frac{[|f|]_{I'_2}}{(b-a)\cdot\! \|f''\|_{C^0(I)}}-\rho.
\end{equation}
Then, applying Theorem~\ref{thm:glob-classic-controllability-convex-1} for
 each pair $\overline u_n, \psi_n\in  C^1([a,b])$, we deduce the existence of 
$\{h_n\}_{n\geq 1}\subset C^0([0,T])$, with $T>T^*$, and $\{u_n\}_{n\geq 1}\subset C^0([a,b]\times [0,T])$ 
that are classical solutions of 
\!\!\!\!\!\!
\begin{align}
\label{eq:cpblaw-n}
&\partial_t u_n+\partial_x f(u_n)=h_n(t)\,,~~~~t\in [0, T],\ x\in [a,b]\,,
\\
\noalign{\medskip}
\label{approx-initial-data-n}
&u_n(0,x)=\overline{u}_n(x) \qquad x\in [a,b]\,, \\
\noalign{\medskip}
\label{approx-terminal-data-n}
& u_n(T,x)=\psi_n(x) \qquad x\in [a,b]\,,
\end{align}
which satisfy the estimates
\begin{equation}
\label{tot-bar-bound111}
 \|h_n\|_{C^0([0,T])}+\text{Tot.\!Var.}\{ h_n;\,[0,T]\}\leq C_1\cdot 
\Big(1+ \|\overline u\|_{C^0([a,b])}+ \|\psi\|_{C^0([a,b])}
\Big),
\end{equation}
 and 
\begin{equation}
\label{tot-bar-bound222}
 \|u_n(t,\cdot)\|_{C^0([a,b])}+\text{Tot.\!Var.}\{ u_n(t,\cdot);\,[a,b]\}\leq C_1\cdot 
\bigg( \|\overline u\|_{C^0([a,b])}+ \|\psi\|_{C^0([a,b])}+\frac{[|f|]_{I'_1}+[|f|]_{I'_2}}{(b-a)\cdot\! \|f''\|_{C^0(I)}}\bigg)
\end{equation}
for all $n\geq 1$, $t\in (0,T)$.

Observe that each $u_n$ is also a weak entropic solution of~\eqref{eq:cpblaw-n}
and that, since  \eqref{tot-bar-bound222} provides a uniform bound on the total variation of $u_n(t,\,\cdot\,)$  for
all $t\in [0,T]$, applying~\cite[Theorem 4.3.1]{DafermosBook} we deduce that 
$t\to u_n(t,\,\cdot)$ is Lipschitz continuous in ${\bf L^1}([a,b])$ on $[0,T]$.
Moreover, by~\eqref{tot-bar-bound222} $\{u_n(t,\,\cdot\,)\}_{n\geq 1}$  are uniformly bounded
for all $t\in [0,T]$.
Therefore, invoking a consequence of Helly's compactness Theorem (e.g. see~\cite[Theorem 2.4]{BressanBook}),
we deduce the existence of a function $u\in{\bf L^1}([0,T]\times[a,b];\, I)$,  which is Lipschitz continuous from $[0,T]$ 
into ${\bf L^1}([a,b];\, I)$, and such that, up to a subsequence, there holds
\begin{equation}
\label{classical-approx-conv-1}
 u_n(t,\,\cdot\,) \to u(t,\,\cdot\,) \qquad
  \text{in}\qquad {\bf L^1} ([a,b])\qquad \forall~t\in [0,T]\,.
\end{equation}
On the other hand  \eqref{tot-bar-bound111} provides a uniform bound on
$\{h_n\}_n$ and on their total variation. Hence, by Helly's compactness Theorem there will be
 a function $h\in \in BV([0,T])$ so that, up to a subsequence, there holds
\begin{equation}
\label{source-approx-conv-1}
 h_n \to h \qquad
  \text{in}\qquad {\bf L^1} ([0,T])\,.
\end{equation}
Hence, relying on~\eqref{classical-approx-conv-1}-\eqref{source-approx-conv-1}, and on the fact that
each $u_n$ is an entropic weak solution of~\eqref{eq:cpblaw-n}, we deduce
\begin{equation}
\label{entr-inew-approx}
\begin{aligned}
&\int_0^T\int_a^b \Big\{\eta(u(t,x))\partial_t\varphi(t,x)+q(u(t,x))\partial_x\varphi(t,x)+\eta'(u(t,x)) h(t)\cdot \varphi(t,x)\Big\}\,dx\,dt\\
\noalign{\smallskip}
& =\lim_{n\to\infty} \int_0^T\int_a^b \Big\{\eta(u_n(t,x))\partial_t\varphi(t,x)+q(u_n(t,x))\partial_x\varphi(t,x)+\eta'(u_n(t,x)) h_n(t)\cdot \varphi(t,x)\Big\}\,dx\,dt\\
& \geq 0\,,
\end{aligned}
\end{equation}
for every entropy / entropy flux pair $(\eta,q)$, with $\eta$ convex. Thus \eqref{entr-inew-approx}, together with~\eqref{initial-terminal-data-conv-1}, \eqref{approx-initial-data-n}, \eqref{classical-approx-conv-1}, proves that $u$ is an entropic weak solution of the Cauchy problem~\eqref{eq:blaw}, \eqref{eq:datum}, while~\eqref{initial-terminal-data-conv-1}, \eqref{approx-terminal-data-n}, \eqref{classical-approx-conv-1} show that the terminal condition~\eqref{terminal-datum} is satisfied. 
Finally, we observe that, by the lower semicontinuity of the total variation with respect to the ${\bf L^1}$
convergence, and because of~
\eqref{classical-approx-conv-1}, \eqref{source-approx-conv-1}, we recover the estimates~\eqref{tot-bar-bound3}, \eqref{tot-bar-bound4}, from \eqref{tot-bar-bound111} and \eqref{tot-bar-bound222}, respectively.
This concludes the proof of the theorem.
\qed
%

\bigskip

\noindent
{\bf Proof of Theorem~\ref{thm:glob-controllability-convex-1b}.}\\
\vspace{-10pt}

\noindent
To fix the ideas assume that the flux $f$ satisfies the assumptions {\bf (H1)} and\, {\bf (H2)-(ii)}.
Then,
given $\overline u\in BV([a,b])$, $\psi\in BV([a,b])$ satisfying~\eqref{cond-iniital-terminal-data-22}, 
and $T>0$, setting $I'_1\doteq \mathrm{Im}(\overline u)$, $I'_2\doteq \mathrm{Im}(\psi)$, $I_u\doteq (i_-,\,u)$, by the same arguments,
and with the same notations, of the proof of Theorem~\ref{C1-null-controll-nonconvex-2}, we deduce that 
there will be $u_0>i_-$ such that
\vspace{-3pt}
\begin{equation}
\begin{gathered}
T> (b-a)\cdot \left(\frac{1}{\,[|f|]_{I'_1,{u_0}}}+ \frac{1}{\,[|f|]_{I'_2,{u_0}}}\right),
\\
\noalign{\medskip}
\sup_{x\in[a,b]} \left\lfloor D^- \overline u(x)\right\rfloor_-<~\frac{[|f|]_{I'_1, u_0}}{(b-a)\cdot\! \|f''\|_{C^0(I)}}-\rho\,,
\qquad\quad
\sup_{x\in[a,b]} \left\lfloor D^+ \overline u(x)\right\rfloor_-<~\frac{[|f|]_{I'_2, u_0}}{(b-a)\cdot\! \|f''\|_{C^0(I)}}-\rho\,,
\end{gathered}
\end{equation}
for some $\rho>0$. 
Hence,
 according to Lemma~\ref{approx-bvmaps}
there exist sequences
$\{\overline u_n\}_{n\geq 1}, \{\psi_n\}_{n\geq 1} \subset C^1([a,b])$, with $\mathrm{Im}(\overline u_n)\subsetneq \mathrm{Im}(\overline u)$, 
$\mathrm{Im}(\psi_n)\subsetneq \mathrm{Im}(\psi)$, 
which satisfy~\eqref{initial-terminal-data-conv-1}, \eqref{cond-iniital-terminal-data-222}.
Now, applying Theorem~\ref{thm:glob-classic-controllability-convex-1} to the flux $f: I_{u_0}\to \R$
which satisfy the assumptions {\bf (H1), (H2)-(i)}, $[|f|]_{I'_1,u_0}>0$, $[|f|]_{I'_2,u_0}>0$,
and to  each pair $\overline u_n, \psi_n\in  C^1([a,b])$, 
that satisfy the estimates~\eqref{bound-initial-terminal-data-2},
by the same arguments of the proof of Theorem~\ref{thm:glob-controllability-convex-1}
we deduce the conclusions of Theorem~\ref{thm:glob-controllability-convex-1b}.
\qed

\section{Some applications}
\label{sec:appl}
In this section we discuss the application of the  controllability results 
established in the paper 
to some examples of conservation laws describing vehicular  traffic
 and sedimentation processes.
Traffic source control can be implemented in a variety of ways
so to modulate the flux capacity
of the road, e.g.~using route  recommendation panels, 
variable speed limit regulation~\cite{GGK}, 
employing integrated vehicular and roadside sensors~\cite{GGBKCS}
or autonomous vehicles~\cite{GGLP}.
Control strategies adopted in the process of 
continuous sedimentation taking place in a clarifier-thickener unit, or settler (used, for
example, in waste water treatment),
usually consist in modulating 
the inflow and outflow of the settler
containing solid particles dispersed in a liquid~\cite{Di1,DF}.
\vspace{-5pt}

\subsection{LWR traffic flow models}
Consider the 
Lighthill, Whitham~\cite{LiWh}  and Richards~\cite{Ri} (LWR) model describing 
the evolution of unidirectional
traffic flow along a stretch of road, say parametrized by $x\in[a,b]$,
given by the conservation law 
\begin{figure}
\begin{center}
\hspace{-1truecm}\resizebox{!}{4truecm}{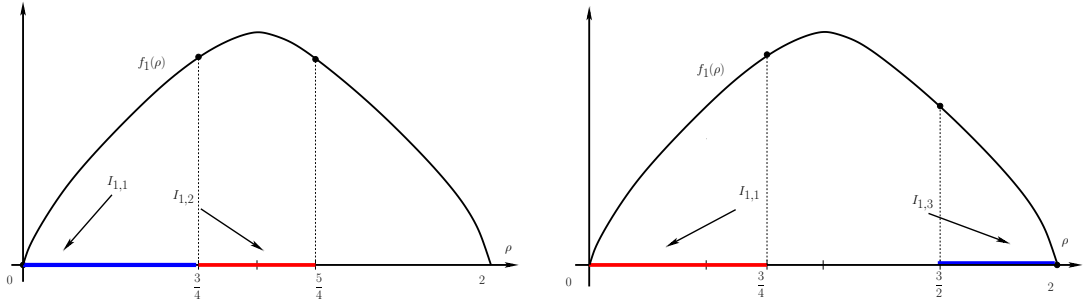}
\caption{Flux $f_1(\rho)=\rho\, (2-\rho)$
\label{lwrconc}}
\end{center}
\end{figure}
\vspace{-3pt}
\begin{equation}\label{LWR} 
\partial_t\rho_t+ \partial_x f(\rho)=0\,,
\end{equation}

\vspace{-10pt}
where  $\rho(t,x)$ denotes the (normalized) traffic density, taking values in the interval $[0,2]$,
and $f(\rho)=\rho\, v(\rho)$ is the flux (the so-called fundamental diagram) depending on 
 the average traffic speed $v(\rho)$.
We first assume that, according  with the 
Greenshields' relationship, $v(\rho)=2-\rho$ which leads to the strictly concave flux 
\vspace{-5pt}
\begin{equation}
f_1(\rho)=\rho\, (2-\rho)\qquad\quad \rho \in [0,2]\,.
\end{equation}
Then, in connection with the sets $I_{1,1} =[0,\frac{3}{4}]$, 
$I_{1,2} =  [\frac{3}{4},\frac{5}{4}]$,
$I_{1,3}=[\frac{3}{2}, 2]$ (see Figure~\ref{lwrconc}), by a direct computation we find
\begin{equation}
[|f_1|]_{I_{1,1}}=f'_1\left({\textstyle\frac{3}{4}}\right)={\textstyle \frac{1}{2}}\,,
\qquad\qquad
[|f_1|]_{I_{1,2}}=\left|\frac{f_1(\frac{3}{2})-f_1(\frac{3}{4})}{\frac{3}{4}}\right|={\textstyle \frac{1}{4}}\,,
\qquad\qquad
[|f_1|]_{I_{1,3}}=\left|f'_1\left({\textstyle \frac{3}{2}}\right)\right|=1\,.
\end{equation}
On the other hand, we have $f''_1(\rho)=-2$.
 Hence, invoking Remark~\ref{concave-convex-1} for the equation~\eqref{LWR} with $f(\rho)=f_1(\rho)$, we deduce that we can produce a source control~$h(t)$ which steers any 
$\overline u\in C^1([a,b])$ to any target profile $\psi\in C^1([a,b])$:
\vspace{-20pt}
\begin{itemize}
\item[--] in a time 
$T>T^*_{1,1}+T^*_{1,2}=6\, (b-a)$, provided that  $\mathrm{Im}(\overline u)\subsetneq I_{1,1}$, $\mathrm{Im}(\psi)\subsetneq I_{1,2}$, and
\vspace{-5pt}
\begin{equation*}
\sup_{x\in [a,b]}\!\lfloor \overline u'(x)\rfloor_+ < \frac{1}{4\,(b-a)} \,,
\quad
\quad\sup_{x\in [a,b]}\!\lfloor  \psi'(x)\rfloor_- < \frac{1}{8\,(b-a)}\,;
\end{equation*}
\item[--] in a time 
$T>T^*_{1,3}+T^*_{1,1}=3\, (b-a)$,
provided that  $\mathrm{Im}(\overline u)\subsetneq I_{1,3}$, $\mathrm{Im}(\psi)\subsetneq I_{1,1}$, and
\vspace{-5pt}
\begin{equation*}
\sup_{x\in [a,b]}\!\lfloor \overline u'(x)\rfloor_+ < \frac{1}{2\,(b-a)} \,,
\quad
\quad\sup_{x\in [a,b]}\!\lfloor  \psi'(x)\rfloor_- < \frac{1}{4\,(b-a)}\,.
\end{equation*}
\end{itemize}
Observe that in the first case we are controlling a state $\overline u$ to a target state $\psi$ with possibly vanishing characteristics
since $f_1'(1)=0$ and $1\in I_{1,2}$. Notice that the choice of the intervals $I_{1,1}, I_{1,2}, I_{1,3}$ is made only to simplify the computation, but one can 
derive similar results for any pair of interval $I'_1, I'_2\subsetneq [0,2]$ such that $[|f_1|]_{I'_1}>0$, $[|f_1|]_{I'_2}>0$
by first solving the optimization problem
related to the definition~\eqref{norm1-[Delta-def} of $[|f_1|]_{I'_i}$, $i=1,2$, and then carrying out similar computations as above.
On the other hand, relying on Remark~\ref{concave-convex-2}
we can produce a source control~$h(t)$ which steers any 
$\overline u\in BV([a,b])$ to any target profile $\psi\in BV([a,b])$:
\vspace{-10pt}
\begin{itemize}
\item[--] in a time 
$T>T^*_{1,1}+T^*_{1,2}=6\, (b-a)$, provided that  $\mathrm{Im}(\overline u)\subsetneq I_{1,1}$, $\mathrm{Im}(\psi)\subsetneq I_{1,2}$, and
\vspace{-5pt}
\begin{equation*}
\sup_{x\in [a,b]}\!\lfloor D^+\overline u(x)\rfloor_+ < \frac{1}{4\,(b-a)} \,,
\quad
\quad\sup_{x\in [a,b]}\!\lfloor D^- \psi(x)\rfloor_- < \frac{1}{8\,(b-a)}\,;
\end{equation*}
\item[--] in a time 
$T>T^*_{1,3}+T^*_{1,1}=3\, (b-a)$,
provided that  $\mathrm{Im}(\overline u)\subsetneq I_{1,3}$, $\mathrm{Im}(\psi)\subsetneq I_{1,1}$, and
\vspace{-5pt}
\begin{equation*}
\sup_{x\in [a,b]}\!\lfloor D^+\overline u(x)\rfloor_+ < \frac{1}{2\,(b-a)} \,,
\quad
\quad\sup_{x\in [a,b]}\!\lfloor D^- \psi(x)\rfloor_- < \frac{1}{4\,(b-a)}\,.
\end{equation*}
\end{itemize}

\begin{figure}
\begin{center}
\resizebox{!}{3truecm}{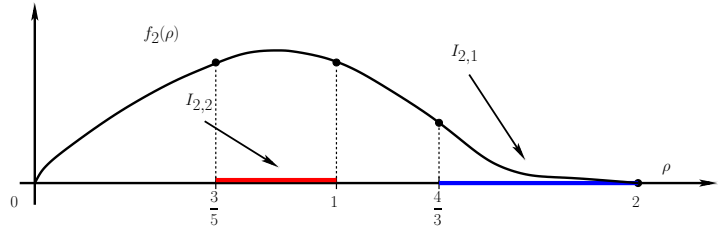}
\caption{Flux $f_2(\rho)=\rho\, e^{-\,\dfrac{\rho}{(2-\rho)}}$
\label{lwrflex}}
\end{center}
\end{figure}
Next, we assume that the traffic speed has the expression $v(\rho)= e^{-\frac{\rho}{(2-\rho)}}$
according with the Bonzani and
Mussone's model~\cite{BM}, 
which leads to the (non concave) bell-shaped flux
\vspace{-5pt}
\begin{equation}
f_2(\rho)=\rho\, e^{-\,\dfrac{\rho}{(2-\rho)}}\qquad\quad \rho \in [0,2]\,.
\end{equation}
Then, in connection with the set $I_{2,1} = [\frac{4}{3}, 2]$, $I_{2,2}= [\frac{3}{5},1]$
(see Figure~\ref{lwrflex}), 
by a direct computation we find
\begin{equation}
[|f_2|]_{I_{2,1}}\approx \left|\frac{f_2(\frac{4}{3})-f_2(\frac{4}{3}-0,717)}{0,717}\right|
\approx 0,298
\qquad
[|f_2|]_{I_{2,2}}=\Big|f_2({\textstyle\frac{8}{5}})-f_2({\textstyle\frac{3}{5}})\Big|\approx 0,361\,.
\end{equation}
On the other hand, we have $\|f''_2\|_{C^0([0,2])}=f''_2(\frac{11+\sqrt{13}}{9}\,)
\approx 2,323$,
and thus 
$\frac{[|f_2|]_{I_{2,1}}}{\|f''_2\|_{C^0([0,2])}}\approx 0,128$,
$\frac{[|f_2|]_{I_{2,2}}}{\|f''_2\|_{C^0([0,2])}}\approx 0,155$.
 Hence, invoking Remark~\ref{concave-convex-1} for the equation~\eqref{LWR} with $f(\rho)=f_2(\rho)$, we can produce a source control~$h(t)$ which steers any 
$\overline u\in C^1([a,b])$ to any target profile $\psi\in C^1([a,b])$:
\vspace{-10pt}
\begin{itemize}
\item[--] in a time 
$T>T^*_{2,1}+T^*_{2,2}\approx 6,125\, (b-a)$, provided that  
$\mathrm{Im}(\overline u)\subsetneq I_{2,1}$, $\mathrm{Im}(\psi)\subsetneq I_{2,2}$, and
\vspace{-5pt}
\begin{equation*}
\sup_{x\in [a,b]}\!\lfloor \overline u'(x)\rfloor_+ < \frac{0,128}{(b-a)}\,,
\quad
\quad\sup_{x\in [a,b]}\!\lfloor  \psi'(x)\rfloor_- <  \frac{0,155}{(b-a)} \,.
\end{equation*}
\end{itemize}
Similarly, relying on Remark~\ref{concave-convex-2},
we can produce a source control~$h(t)$ which steers any 
$\overline u\in BV([a,b])$ to any target profile $\psi\in BV([a,b])$:
\vspace{-10pt}
\begin{itemize}
\item[--] in a time 
$T>T^*_{2,1}+T^*_{2,2}\approx 6,125\, (b-a)$, provided that  
$\mathrm{Im}(\overline u)\subsetneq I_{2,1}$, $\mathrm{Im}(\psi)\subsetneq I_{2,2}$, and
\vspace{-5pt}
\begin{equation*}
\sup_{x\in [a,b]}\!\lfloor D^+\overline u(x)\rfloor_+ < \frac{0,128}{(b-a)}\,,
\quad
\quad\sup_{x\in [a,b]}\!\lfloor D^- \psi(x)\rfloor_- < \frac{0,155}{(b-a)} \,.
\end{equation*}
\end{itemize}
\vspace{-5pt}
Again, we observe that these results guarantee the controllability of possibly critical states since $f'_2(3-\sqrt{5})=0$
and $3-\sqrt{5}\in I_{2,2}$.

\subsection{Kynck's sedimentation model}
According with the solid-flux theory by Kynch~\cite{Ky},  
the sedimentation of a suspension of small particles dispersed in a viscous fluid
can be described by a  conservation law
\begin{equation}\label{Ky} 
\partial_t u_t+ \partial_x f(u)=0\,,
\end{equation}
where $u(t,x)$ denotes the solid fraction, taking values in the interval $[0,1]$,
and the flux function (also called drift-flux)
has the same type of expression of the LWR flux, i.e. 
$f(u)=u\, v(u)$, with $v(u)$ denoting the local settling velocity of the particles.
Typically $f$ is a concave-convex map with one inflection
point. Here we consider the 
sedimentation of a solid substance suspended in a cylindrical
batch of height $L$, parametrized so that the bottom is located at $x=0$
and the top at $x=L$, with the drift-flux function proposed in~\cite{MW}
which, up to normalization, in this case can be written as 
\begin{equation}
f_3(u)=-u\, (1-u)^2\qquad\ u \in [0,1]\,.
\end{equation}
\begin{figure}
\begin{center}
\hspace{-2truecm}\resizebox{!}{3truecm}{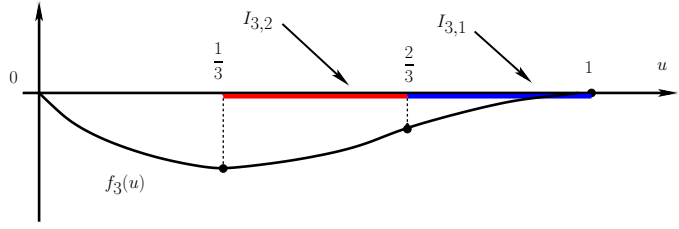}
\caption{Flux $f_3(u)=-u\, (1-u)^2$
\label{kynch}}
\end{center}
\end{figure}
Then, in connection with the set $I_{3,1} = [\frac{2}{3}, 1]$, $I_{3,2}=[\frac{1}{3},\frac{2}{3}]$
(see Figure~\ref{kynch}), 
by a direct computation we find
\begin{equation}
[|f_3|]_{I_{3,1}}=
[|f_3|]_{I_{3,2}}=\left|\frac{f_3(1)-f_3(\frac{2}{3})}{\frac{1}{3}}\right|=\left|\frac{f_3(\frac{2}{3})-f_3(\frac{1}{3})}{\frac{1}{3}}\right|=\frac{2}{9}
\approx 0,222\,.
\end{equation}
On the other hand, we have $\|f''_3\|_{C^0([0,1])}=|f''_3(0)|=4$,
and thus $\frac{[|f_3|]_{I_{3,1}}}{\|f''_3\|_{C^0([0,1])}}=\frac{[|f_3|]_{I_{3,2}}}{\|f''_3\|_{C^0([0,1])}}\approx 0,055$.
 Hence, invoking Remark~\ref{concave-convex-1} for the equation~\eqref{Ky} with $f(u)=f_3(u)$, we can produce a source control~$h(t)$ which steers any 
$\overline u\in C^1([a,b])$ to any target profile $\psi\in C^1([a,b])$:
\vspace{-10pt}
\begin{itemize}
\item[--] in a time 
$T>T^*_{3,1}+T^*_{3,2}=9\, (b-a)$, provided that  
$\mathrm{Im}(\overline u)\subsetneq I_{3,1}$, $\mathrm{Im}(\psi)\subsetneq I_{3,2}$, and
\vspace{-5pt}
\begin{equation*}
\sup_{x\in [a,b]}\!\lfloor \overline u'(x)\rfloor_+ < \frac{1}{2\,(b-a)} \,,
\quad
\quad\sup_{x\in [a,b]}\!\lfloor  \psi'(x)\rfloor_- < \frac{1}{2\,(b-a)}\,.
\end{equation*}
\end{itemize}
Similarly, relying on Remark~\ref{concave-convex-2},
we can produce a source control~$h(t)$ which steers any 
$\overline u\in BV([a,b])$ to any target profile $\psi\in BV([a,b])$:
\vspace{-10pt}
\begin{itemize}
\item[--] in a time 
$T>T^*_3=9\, (b-a)$, provided that  $\mathrm{Im}(\overline u)\subsetneq I_{3,1}$, $\mathrm{Im}(\psi)\subsetneq I_{3,2}$, and
\vspace{-5pt}
\begin{equation*}
\sup_{x\in [a,b]}\!\lfloor D^+\overline u(x)\rfloor_+ < \frac{1}{2\,(b-a)} \,,
\quad
\quad\sup_{x\in [a,b]}\!\lfloor D^- \psi(x)\rfloor_- < \frac{1}{2\,(b-a)}\,.
\end{equation*}
\end{itemize}
\vspace{-5pt}
Again, we observe that these results guarantee the controllability of possibly critical states since $f'_3(\frac{1}{3})=0$
and $\frac{1}{3}\in I_{3,1}$.

\section{Appendix}
\label{sec:app}

The approximation of BV function satisfying a one-sided Lipschitz condition in
terms of smooth functions satisfying the same Lipschitz condition (used in the proof of Theorem~\ref{thm:glob-controllability-convex-1})
is guaranteed by the following lemma. The result is standard, but we provide a proof for completness.
\begin{lemma}
\label{approx-bvmaps}
Let $\varphi \in BV([a,b])$, with $\text{Im}(\varphi)\subsetneq I$, satisfy
\begin{equation}
\label{cond-D+1}
D^+\varphi(x) < M\qquad \forall~x\in[a,b]\,,
\end{equation}
for some $M>0$. Then, there exists $\{\varphi_n\}_{n\geq 1}\subset  C^1([a,b])$, 
with  
$\text{Im}(\varphi_n)\subseteq  I$, 
for all $n$ sufficiently large,
and  satisfying
\begin{equation}
\label{cond-D+2}
 \varphi'_n(x) < M\qquad \forall~x\in[a,b]\,,\quad\forall~n\geq 1\,,
\end{equation}
such that
\begin{equation}
\varphi_n\to \varphi\quad \text{in}\quad {\bf L^1} ([a,b])\,.
\end{equation}
\end{lemma}
\begin{proof}  {Observe that, because of~\eqref{cond-D+1}, the map $x\mapsto \psi(x)=\varphi(x)-Mx$ is strictly decreasing on $[a,b]$. Let $\rho_n\in C_c^\infty(\R)$, $n>0$, be a standard mollifier, i.e.
\[
\rho_n~\geq~0,\qquad\mathrm{sup}(\rho_n)~\subseteq~ (-{\textstyle{\frac{1}{n}}},\,{\textstyle{\frac{1}{n}}}),\qquad\mathrm{and}\qquad \int_{\R}\rho_n(x)dx~=~1.
\]
Then, we have that $\psi_n=\rho_n*\psi\in C^\infty([a,b])$, with $Im(\psi_{n})\subsetneq  I$
for all $n$ sufficiently large, and 
\[
\psi_n\to \psi\quad \text{in}\quad {\bf L^1} ([a,b])\,.
\]
Moreover, for every $x_1<x_2$, there holds 
\[
\psi_{n}(x_2)-\psi_{n}(x_1)~=~\int [\psi(x_2-y)-\psi(x_1-y)]\cdot \rho_{n}(y)dy~<~0.
\]
Thus, one has $D^+\psi_n(x) <0$ for all $n$,
and the sequence $\varphi_n=\psi_n+Mx$ does the job.
}
%
\end{proof}

\end{document}